\documentclass[a4paper,reqno]{amsart}

\usepackage[utf8]{inputenc}
\usepackage[T1]{fontenc}
\usepackage{lmodern}
\usepackage[english]{babel}
\usepackage{microtype,hyperref}

\usepackage{amsmath,amssymb,amsfonts,amsthm}
\usepackage{mathtools,accents, upgreek,xcolor}
\usepackage{mathrsfs,comment,aliascnt,braket,bm,esint}
\usepackage{graphicx} % \scalebox

\usepackage{etoolbox,comment}
\makeatletter
\patchcmd{\@maketitle}
  {\ifx\@empty\@dedicatory}
  {\ifx\@empty\@date \else {\vskip3ex \centering\footnotesize\@date\par\vskip1ex}\fi
   \ifx\@empty\@dedicatory}
  {}{}
\patchcmd{\@adminfootnotes}
  {\ifx\@empty\@date\else \@footnotetext{\@setdate}\fi}
  {}{}{}
\makeatother

\newcommand{\bC}{\mathbb{C}}

\newcommand{\bR}{\mathbb{R}}

\newcommand{\cC}{\mathcal{C}}

\newcommand{\dist}{\mathsf{d}}

\newcommand{\vol}{\mathrm{vol}}
\newcommand{\eps}{\ensuremath{\varepsilon}}

\theoremstyle{theorem}
\newtheorem{theorem}{Theorem}

\theoremstyle{definition}
\newtheorem{bump}{Bump}[section]

\newtheorem{lemma}[bump]{Lemma}
\newtheorem{definition}[bump]{Definition}
\newtheorem{corollary}[bump]{Corollary}
\newtheorem{prop}[bump]{Proposition}
\newtheorem{rem}[bump]{Remark}

\theoremstyle{remark}
\newtheorem{claim}{Claim}

\newcommand{\df}{\coloneqq}
\newcommand{\fd}{\eqqcolon}
\newcommand{\di}{\mathop{}\!\mathrm{d}}

\DeclareMathOperator{\sinc}{sinc}
\DeclareMathOperator{\Lip}{Lip}

\mathtoolsset{showonlyrefs}

\title[Spectral properties of symmetrized AMV operators]{Spectral properties \\ of symmetrized AMV operators}

\author{Manuel Dias}
\address{M.~Dias: Department of Mathematics and Data Science, Vrije Universiteit Brussel, Pleinlaan 2, B-1050 Elsene, Belgium.}
\email{manuel.dias@vub.be}

\author{David Tewodrose}
\address{D.~Tewodrose: Department of Mathematics and Data Science, Vrije Universiteit Brussel, Pleinlaan 2, B-1050 Elsene, Belgium.}
\email{david.tewodrose@vub.be}

\date{\today}

\begin{document}

\maketitle

 \begin{abstract}
The symmetrized Asymptotic Mean Value Laplacian $\tilde{\Delta}$, obtained as limit of approximating operators $\tilde{\Delta}_r$, is an extension of the classical Euclidean Laplace operator to the realm of metric measure spaces. We show that, as $r \downarrow 0$, the operators $\tilde{\Delta}_r$ eventually admit isolated eigenvalues defined via min-max procedure on any compact  uniformly locally doubling metric measure space. Then we prove $L^2$ and spectral convergence of $\tilde{\Delta}_r$ to the Laplace--Beltrami operator of a compact Riemannian manifold, imposing Neumann conditions when the manifold has a non-empty boundary.
\end{abstract}

\tableofcontents

\section{Introduction}

In the past thirty years, much research has been carried out to extend the classical Euclidean Laplace operator to metric measure spaces: see e.g.~\cite{CC,KMS,Gigli,AB}. This paper deals with such an extension, namely the symmetrized Asymptotic Mean Value (AMV) Laplacian, proposed in \cite{MT2}, see also \cite{K,MT1,AKS1,AKS2}. The symmetrized AMV Laplacian is set as
\begin{equation}\label{eq:conv}
\tilde{\Delta} \df \lim\limits_{r \downarrow 0} \tilde{\Delta}_r
\end{equation}
where for $\mu$-a.e.~$x \in X$,
\[
\tilde{\Delta}_r f(x) \df \frac{1}{2r^2} \fint_{B_r(x)} \left( 1 + \frac{V(x,r)}{V(y,r)} \right) (f(y)-f(x)) \di \mu(y).
\]
Here $f$ is a locally integrable function defined on a metric measure space $(X,\dist,\mu)$. Throughout the paper, $B_r(z)$ denotes the metric open ball centered at $z \in X$ with radius $r>0$, the notation $V(z,r)$ stands for $\mu(B_r(z))$, and $\fint_{B_r(z)}$ is shorthand for $V(z,r)^{-1} \int_{B_r(z)}$.

Part of the study on the symmetrized AMV Laplacian consists in finding a relevant meaning to the limit in \eqref{eq:conv}. If this is intended in the $L^2$ sense, then the associated spectral convergence can be investigated. This is the goal of the present paper. For any $k \in \mathbb{N}$, set
\begin{equation}
\tilde{\lambda}_{k,r} \df  \inf_{V \in \mathcal{G}_{k+1}(L^2(X,\mu))}
    \sup_{f \in V}
    \frac{\tilde{E}_r(f)}
    {\|f\|_2},
\end{equation}
where $\mathcal{G}_{k+1}(L^2(X,\mu))$ is the $(k+1)$-th Grassmannian of $L^2(X,\mu)$, and $\tilde{E}_r(f)$ is the energy functional naturally associated with $\tilde{\Delta}_r$ (Definition \ref{def:min-max}). These form a non-decreasing sequence of non-negative numbers. Our first main result states that these numbers eventually correspond to isolated eigenvalues of $-\tilde{\Delta}_r$ when $(X,\dist,\mu)$ is compact and  uniformly locally doubling (Definition \ref{def:loc_doubling}).

\begin{theorem}\label{th:1}
    Let $(X,\dist,\mu)$ be a compact uniformly locally doubling metric measure space. For any integer $k \ge 2$, there exists $r_k>0$ such that for any $r \in (0,r_k)$, the operator $-\tilde{\Delta}_r$ admits $k+1$ eigenvalues $$ 0 = \lambda_0(-\tilde{\Delta}_r) < \lambda_1(-\tilde{\Delta}_r) \le  \ldots \le \lambda_k(-\tilde{\Delta}_r)$$ such that $\lambda_i(-\tilde{\Delta}_r)=\tilde{\lambda}_{i,r}$ for any $i \in \{0,\ldots,k\}$.
\end{theorem}

Our second main result deals with a smooth manifold $M$ endowed with a smooth Riemannian metric $g$. We write $\Delta_g$ for the (negative) Laplace--Beltrami operator of $(M,g)$. We let $m \ge 2$ be the dimension of $M$, and we set
\begin{equation}\label{eq:C_m}
C_m \df \frac{1}{2}\fint_{\mathbb{B}^m_1(0)}
            \xi_1^2
        \di \xi
    =
        \frac{1}{2(m+2)} 
\end{equation}
where $\mathbb{B}^m_1(0)$ is the unit Euclidean ball of $\mathbb{R}^m$. In this context, it follows from the equality between symmetrized and non-symmetrized AMV Laplacian and a simple calculation in normal coordinates that
\begin{equation}\label{eq:pointwise}
\tilde{\Delta}_r f (x) \stackrel{r\downarrow 0}{\to} C_m \Delta_g f(x)
\end{equation}
for any $f \in \cC^2(M)$ and any interior point $x \in M$, see \cite{MT1,MT2} --- the convergence is even locally uniform in the interior of $M$, see \cite{AKS1}. We refer to \cite{MT1,MT2,AKS1,AKS2} for related pointwise results in various settings like Carnot groups or Alexandrov spaces.

In this paper, we are interested in the $L^2$ version of \eqref{eq:pointwise} with a particular interest in the case where $M$ admits a non-empty boundary $\partial M \neq 0$. In this case, we write $\partial_\nu f \in \cC^\infty(\partial M)$ for the normal derivative of a smooth function $f : M \to \mathbb{R}$, and we define
\begin{equation}\label{eq:Cnu}
\cC^\infty_{\nu}(M) \df \{f \in \cC^\infty(M) \, : \, \partial_\nu f =0 \}.
\end{equation}
We see $(M,g)$ as a metric measure space $(M,\dist_g,\vol_g)$ where $\dist_g$ and $\vol_g$ are the Riemannian distance and volume measure on $M$ associated with $g$. Then our statement reads as follows.

\begin{theorem}\label{th:2}
Let $(M^m,g)$ be a compact, connected, smooth Riemannian manifold with a non-empty (resp.~empty) boundary $\partial M$. Then for any $f \in \cC_\nu^\infty(M)$ (resp.~$\cC^\infty(M)$),  as $r \downarrow 0$,
        \[
        \tilde{\Delta}_r f \stackrel{L^2}{\longrightarrow} C_m \Delta_g f.
        \]
\end{theorem}
We point out that the boundaryless version of this result is rather easy to obtain, while a non-empty boundary is quite tricky to handle. The Neumann condition in the latter case is crucial to ensure convergence: indeed, the sequence $\tilde{\Delta}_r f$ may blow-up if this is not imposed.

After the previous $L^2$-convergence result, we address the question of spectral convergence, that is to say, the convergence of the associated eigenvalues and eigenfunctions. In this regard we show that, for any $k \in \mathbb{N}$, the function $r \mapsto \tilde{\lambda}_{k,r}$ is bounded in a neighborhood of $0$, as proved in the course of Theorem \ref{th:1}. This ensures that the $k$-th lowest eigenvalue of the operator $-\tilde{\Delta}_r$, which we denote $\lambda_k(-\tilde{\Delta}_r)$, exists for small enough $r$, and that it coincides with $\tilde{\lambda}_{k,r}$. Let $f_{k,r}$ be an $L^2$-normalized eigenfunction of $-\tilde{\Delta}_r$ associated with $\lambda_k(-\tilde{\Delta}_r)$. Recall that if $\partial M = \emptyset$ (resp.~$\partial M \neq \emptyset$), a Laplace (resp.~Neumann) eigenvalue of $(M,g)$ is a number $\mu \ge 0$ for which there exists an associated eigenfuction $f \in \cC^\infty(M)$ (resp.~$\cC_\nu^\infty(M)$) of $-\Delta_g$, i.e.~$-\mu f = \Delta_g f$.

\begin{theorem}\label{th:3}
    Let $(M^m,g)$ be a compact, connected, smooth Riemannian manifold. Assume that $\partial M = \emptyset$ (resp.~$\partial M \neq \emptyset$). For $k \in \mathbb{N}$, let $\mu_k$ be the $k$-th lowest Laplace (resp.~Neumann) eigenvalue of $\Delta_g$. For any $(r_n) \subset (0,+\infty)$ such that $r_n \to 0$, there exists an $L^2$-normalized Laplace (resp.~Neumann) eigenfunction $f \in \cC^\infty(M)$ (resp.~$\cC_\nu^\infty(M)$) associated with $\mu_k$  such that, up to a subsequence,
    \begin{equation}
    \begin{cases}
    \lambda_{k}(-\tilde{\Delta}_{r_n}) \to C_m \,  \mu_k ,\\
    f_{k,r_n} \stackrel{L^2}{\longrightarrow} f.
    \end{cases}
    \end{equation}
\end{theorem}

We point out that the question of spectral convergence for the Gaussian approximation of the Laplace--Beltrami operator of a compact Euclidean submanifold with boundary  was raised in \cite{Belkin+}. This has been one motivation for the present work: to study this convergence with the intrinsic approximation provided by the symmetrized AMV operators $\tilde{\Delta}_r$ instead of the extrinsic Gaussian one.

\subsubsection*{Acknowledgments.} Both authors are funded by the Research Foundation – Flanders (FWO) via the Odysseus II programme no.~G0DBZ23N. They thank the anonymous reviewers for their numerous and constructive suggestions, and Dmitri Pavlov for clarifying a point related to Lemma 2.1.

\section{Averaging-like operators}

In this section, we consider a fixed metric measure space, that is to say, a triple $(X,\dist,\mu)$ where $(X,\dist)$ is a metric space and $\mu$ is a fully supported  regular Borel measure on $(X,\dist)$ such that \[V(x,r) \df \mu(B_r(x))<+\infty\] for any $x \in X$ and $r>0$, where $B_r(x)$ denotes the open ball $\{y \in X : \dist(x,y)<r\}$. Notice that for any $x \in X$ and $r>0$, $$V(x,r)>0,$$ because $\mu$ is fully supported. Moreover, if $X$ is compact, then
\begin{equation}
\mu(X)<+\infty
\end{equation}
since $\mu$ is finite on any ball of radius the diameter of $X$. We set
\begin{equation*}\label{eq:mM}
0 \le m(r) \df \inf_{x \in X} V(x,r)  \le M(r) \df \sup_{x \in X} V(x,r) \le +\infty.
\end{equation*}

Note that our assumptions yield the following preliminary result.

\begin{lemma}\label{lem:separable}
    $L^2(X,\mu)$ is separable.
\end{lemma}
\begin{proof}
    We start by proving that $(X,\dist)$ is a second countable space. Fix $o \in X$. Given $\epsilon>0$ and $N \in \mathbb{N}$ positive, consider the value given by
    \begin{equation}
    \label{eq:supremumOfMeasureOfCovers}
        \alpha_{\epsilon, N}
    =
        \sup \{\left. \mu \right|_{B_N(o)}(\cup_n B_{\epsilon}(x_n)) : \{x_n\}_{n \in \mathbb{N}} \subset X\},
    \end{equation}
    where $\left. \mu \right|_{B_N(o)} (\cdot) \df \mu (\cdot \cap B_N(o))$. First we show this supremum is attained. Consider $\delta_k \rightarrow 0$ and let $\{x_n^k\}_{n \in \mathbb{N}} \subset X$ such that
    \begin{equation}
        \left. \mu \right|_{B_N(o)}( \cup_n B_\epsilon(x_n^k))
        >
        \alpha - \delta_k.
    \end{equation}
    Taking
    \begin{equation}
        \{y_n\}_{n \in \mathbb{N}}
    =
        \cup_k \{x_n^k\}_{n \in\mathbb{N}}
    \end{equation}
    we have that
    \begin{equation}
        \left. \mu \right|_{B_N(o)}(\cup_n B_\epsilon(y_n)) = \alpha_{\epsilon, N}.
    \end{equation}
    Now we prove $\alpha_{\epsilon, N} = \mu(B_N(z)) < \infty$. If $\alpha_{\epsilon, N} < \mu(B_N(z))$, then $\mu(B_{N}(z) \backslash \cup_n B_{\epsilon}(y_n)) > 0$, where $\{y_n\}_{n \in \mathbb{N}}$ is a maximizer of \eqref{eq:supremumOfMeasureOfCovers}. Since the measure is inner regular, there must exist some compact $K\subset B_{N}(z) \backslash \cup_n B_{\epsilon}(y_n)$ such that
    \begin{equation}
        \mu(K)>0.
    \end{equation}
    Since we can cover $K$ by a finite number of balls $B_\epsilon(z_k)$, there must exist some $z = z_k$ such that
    \begin{equation}
        \mu(B_\epsilon(z) \cap K) > 0.
    \end{equation}
    We then have
    \begin{align}
        \left. \mu \right|_{B_N(o)}\left(\cup_n B_\epsilon(y_n) \cup B_\epsilon(z)\right)
    &\geq
        \left. \mu \right|_{B_N(o)}\left(\cup_n B_\epsilon(y_n) \cup \left(B_\epsilon(z) \cap K\right)\right)\\
    &=
        \mu\left(\cup_n B_\epsilon(y_n)\right)
    +
        \mu\left( B_\epsilon(z) \cap K\right)
    >
    \alpha_{N,\epsilon}.
    \end{align}
    And so $\{y_n\}_{n \in \mathbb{N}} \cup \{z\}$ contradicts the maximality of $\{y_n\}_{n \in \mathbb{N}}$. This shows that $\alpha_{\epsilon, N} = \mu(B_N(o))$. 
    Since $\cup_n B_\epsilon(y_n) \cap B_N(o)$ has full measure in the support $B_N(o)$ of $\mu|_{B_N(o)}$, it is a dense subset of $B_N(o)$. This implies that $\cup_n B_{2\epsilon}(y_n)$ is a countable cover of $B_N(o)$. To build a countable basis of $X$, consider a sequence $\delta_k \rightarrow 0$. For any $k$, take $\cup_{n \in \mathbb{N}} B_{\delta_k} (y_n^{k,N})$ a countable cover of $B_N(z)$. Then the set given by
    \begin{equation}
        \mathcal{B}
    =
        \cup_{k,N \in \mathbb{N}}
        \{
            B_{\delta_k}(y_n^{k,N})\}_{n \in \mathbb{N}}
    \end{equation}
    is a countable basis of $X$. Given this basis $\mathcal{B}$ we can create a new basis given by the finite union of elements of $\mathcal{B}$, and we call this new basis $\mathcal{B}'$ which will also be countable. In particular we have that given some open set $V \subset X$ we can find a sequence $V_n$ such that
    \begin{equation}
        V_n \subset V_{n+1}, \quad\quad
        \cup_n V_n = V.
    \end{equation}
    This is the case since $\mathcal{B}$ is a countable basis, we can find elements $B_k \in \mathcal{B}$ such that
    \begin{equation}
        \cup_k B_k = V.
    \end{equation}
    We conclude by taking $V_n = \cup_{k=1}^n B_k \in \mathcal{B}'$. To construct our dense subset of $L^2(X,\mu)$ we take finite sums with rational coefficients of the characteristic functions $\chi_V$, with $V \in \mathcal{B}'$. To show that this is dense in $L^2(X,\mu)$ we only need to show that we can approximate arbitrarily well simple functions $\chi_U$ where $U \subset X$ is open and $\mu(U) < \infty$ since the measure $\mu$ is outer regular. Given such an open set $U$, take $\chi_{U_n}$ where $U_n \in \mathcal{B}'$ and $U_{n}\subset U_{n+1}$ and $\cup_n U_n = U$. Then we have by dominated convergence
    \begin{equation}
        \chi_{U_n}
    \rightarrow_{L^2(X,\mu)}
        \chi_U,
    \end{equation}
    concluding the proof.
\end{proof}

\subsection{Averaging operator}

For any $x,y \in X$ and $r>0$, set
\[
a_r(x,y) \df \frac{1_{B_r(x)}(y)}{V(x,r)}\, \cdot
\]
Consider $u \in L^1_{\text{loc}}(X,\mu)$. For any $x \in X$ and $r>0$ such that $u$ is $\mu$-integrable on $B_r(x)$, set
\[
A_ru(x) \df \fint_{B_r(x)} u \di \mu = \int_X a_r(x,y) u(y) \di \mu(y).
\]
Notice that, since $u$ is locally integrable, for any $x \in X$ there exists $r_x >0$ such that $A_{r_x}u(x)$ is well-defined. However, there may be no uniform $r>0$ for which the integral $A_ru(x)$ is well-defined for every $x \in X$.

Let us also set
\[
a^*_r(x,y) \df a_r(y,x) = \frac{1_{B_r(x)}(y)}{V(y,r)}
\]
for any $x,y \in X$ and $r>0$. Consider $u \in L^0(X,\mu)$ such that $v(\cdot) \df u(\cdot)/V(\cdot,r) \in L^1_{\text{loc}}(X,\mu)$. For any $x \in X$ and $r>0$ such that $v$ is $\mu$-integrable on $B_r(x)$, set
\[
A^*_ru(x) \df \int_{B_r(x)} \frac{u(y) \di \mu(y)}{V(y,r)} = \int_{X} a_r^*(x,y) u(y) \di \mu(y).
\]
Notice that, just like $A_ru(x)$, $A^*_ru(x)$ may not make sense uniformly with respect to $x \in X$. 

For any $r>0$,  we introduce the following conditions:
\begin{equation}\label{I}\tag{$\mathrm{I}_r$}
        \|A_r^*1\|_\infty < +\infty,
\end{equation}
\begin{equation}\label{eq:V-1_int}\tag{$\mathrm{II}_r$}
V(\cdot,r)^{-1} \in L^1(X,\mu).
\end{equation}
Note that \eqref{eq:V-1_int} implies \eqref{I} since
\[
 \|A_r^*1\|_\infty = \sup_{x \in X} |A_r^*1(x)|  = \sup_{x \in X} \int_{B_r(x)} \frac{\di \mu(y)}{V(y,r)} \le \int_X \frac{\di \mu(y)}{V(y,r)} \, \cdot
\]

In the next lemma, we discuss the boundedness and the compactness of the averaging operator $A_r$ acting on Lebesgue spaces.

\begin{lemma}\label{lem:average}
Assume that there exists $r>0$ such that \eqref{I} holds. Then for any $p \in [1,+\infty]$ the linear operator $A_r : L^p(X,\mu) \to L^p(X,\mu)$ is well-defined and bounded with
\[
\| A_r \|_{p \to p} \le \|A_r^*1\|_\infty^{1/p} \cdot
\]
Moreover, if \eqref{eq:V-1_int} holds, then $A_r : L^2(X,\mu) \to L^2(X,\mu)$ is compact.
\end{lemma}

\begin{proof}
The case $p=+\infty$ is obvious and holds regardless of \eqref{I}.  Let us assume that $p<+\infty$. Let $u \in L^p(X,\mu)$. By Jensen's inequality, for any $x \in X$,
    \[
    |A_ru(x)|^p \le \left( \fint_{B_r(x)} |u|^p \di \mu\right).
    \]
   Thus
    \begin{align*}
   \|A_ru\|_{p}^p & \le \int_X \frac{1}{V(x,r)} \int_{B_r(x)} |u(y)|^p \di \mu(y) \di \mu (x)\\
    & = \int_X \int_X \frac{1}{V(x,r)} \underbrace{1_{B_r(x)}(y)}_{=1_{B_r(y)}(x)} |u(y)|^p \di \mu(y) \di \mu (x)\\
    & = \int_X |u(y)|^p \underbrace{\int_{B_r(y)} \frac{ \di \mu(x)}{V(x,r)} }_{=A_r^*1(y)} \di \mu (y) \le \|A_r^*1\|_\infty\|u\|_p^p
    \end{align*}
    where we have used the Fubini--Tonelli theorem to get the second equality and \eqref{I} for the last inequality.

    Let us now assume that \eqref{eq:V-1_int} holds. Since \[
\int_X \int_X a_r^2(x,y) \di \mu(y) \di \mu(x) = \int_X 
\frac{1}{V(x,r)}\fint_{B_r(x)} \di \mu(y) \di \mu(x) = \int_X \frac{\di \mu(x)}{V(x,r)}
\]
we obtain that $A_r$ is a Hilbert-Schmidt integral operator acting on the separable space $L^2(X,\mu)$ (recall Lemma \ref{lem:separable}); in particular,  $A_r$ is compact \cite[Section IV.6]{RS}.
\end{proof}

In the next statement, we provide an alternative way to prove the compactness of $A_r$ from $L^2(X,\mu)$ to itself. This goes through the compactness of $A_r$ from $L^2(X,\mu)$ to the space of continuous functions $\cC(X)$ which we obtain for compact spaces $X$ satisfying  the following condition:
\begin{equation}\label{eq:vanishing_spheres}\tag{$\mathrm{S}_r$}
\sup_{x \in X} \mu( S_r(x)) = 0
\end{equation}
where $S_r(x)\df \{y \in X : \dist(x,y)=r\}$.

\begin{lemma}\label{lem:reg_comp}
Assume that $(X,\dist,\mu)$ is compact and satisfies \eqref{eq:vanishing_spheres} for some $r>0$. Then $A_r : L^2(X,\mu) \to \cC(X)$ is compact and satisfies
\begin{equation}\label{eq:bound}
    \|
        A_r
    \|_{2\to \infty}
\leq
    \frac{1}{m(r)^{1/2}} \, \cdot
\end{equation}
\end{lemma}

\begin{proof}
We start by noticing that if $u \in L^2(X,\mu)$, then $A_r(u)$ is continuous. This follows from $V(\cdot, r)^{-1}$ and $\int_{B_r(x)}u(y)d\mu(y)$ being continuous. The former holds by assumption. To prove the latter, assume that $\|u\|_{L^2(X)} = 1$. Then for any $x, z \in X$,
\begin{align}
    \left|
        \int_{B^r(x)}u(y)d\mu(y)
    -
        \int_{B^r(z)}u(y)d\mu(y)
    \right|
    &\leq
        \left\|
            1_{B_r(x)}
        -
            1_{B_r(z)}
        \right\|_{L^2(X)}
        \left\|
            u
        \right\|_{L^2(X)}\\
    &\leq
        \mu(B_{r+d(x,z)}(x)-B_{r-d(x,z)}(x))^{\frac{1}{2}},
    \label{equiContIneq}
\end{align}  
and $\mu(B_{r+d(x,z)}(x)-B_{r-d(x,z)}(x)) \to 0$ as $\dist(x,z) \to 0$ due to  \eqref{eq:vanishing_spheres}. Moreover, the bound \eqref{eq:bound} is obtained via Hölder's inequality: for any $x \in X$,
\[
|A_ru(x)| \le \left( \fint_{B_r(x)} u^2 \di \mu\right)^{1/2} \le \frac{1}{m(r)^{1/2}} \, \cdot
\]
To prove compactness, consider $\{f_n\} \subset L^2(X, \mu)$ such that $\sup_n \|f_n\|_{2} \leq 1$. Uniform boundedness of $\{A_r(f_n)\}$ follows from \eqref{eq:bound}, and equicontinuity can be obtained by using the inequality \eqref{equiContIneq} applied to the sequence. By the Ascoli-Arzelà Theorem, we can extract from $\{A_r(f_n)\}$ a subsequence which converges in $\mathcal{C}(X)$, concluding the proof.
\end{proof}

\subsection{Adjoint} Let us focus now on the boundedness and the compactness of the adjoint operator $A_r^*$. We begin with a simple observation.

\begin{lemma}\label{lem:L1}
    The operator $A_r^* : L^1(X,\mu) \to L^1(X,\mu)$ is a contraction for any $r>0$.
\end{lemma}\label{rmk:L1}

\begin{proof}
For any $u \in L^1(X,\mu)$,
\begin{align*}
    \int_X |A_r^*u(x)| \di \mu(x) & \le \int_X \int_{B_r(x)} \frac{|u(y)|}{V(y,r)} \di \mu(y) \di \mu(x)\\
    & =  \int_X \int_X 1_{B_r(x)}(y) \frac{|u(y)|}{V(y,r)} \di \mu(y) \di \mu(x)\\
    & =  \int_X \left( \int_X 1_{B_r(y)}(x)  \di \mu(x) \right)\frac{|u(y)|}{V(y,r)}\di \mu(y) = \int_X |u(y)| \di \mu(y),
\end{align*}
where we used the Fubini--Tonelli theorem to get the penultimate equality.
\end{proof}

We continue with the next lemma which covers the case $p>1$.

\begin{lemma}\label{lem:adjoint}
Assume that there exists $r>0$ such that \eqref{I} holds. Then for any $p \in [1,+\infty]$, the linear operator $A_r^* : L^p(X,\mu) \to L^p(X,\mu)$ is well-defined and bounded with
    \[
    \| A_r^* \|_{p \to p} \le \|A_r^*1\|_\infty^{(p-1)/p}
    \]
Moreover, this operator is the adjoint of $A_r: L^q(X,\mu) \to L^q(X,\mu)$ for $q \in [1,+\infty]$ such that $1/p + 1/q = 1$. Lastly,  if \eqref{eq:V-1_int} holds, then  the operator $A_r^* : L^2(X,\mu) \to L^2(X,\mu)$ is compact.
\end{lemma}

\begin{proof}
    For the proof of the first assertion, consider $u \in L^\infty(X,\mu)$. Thanks to \eqref{I}, for $\mu$-a.e.~$x \in X$,
    \begin{align*}
     |A_r^*u|(x)  & \le \int_{B_r(x)} \frac{|u(y)|}{V(y,r)} \di \mu(y)  \le \|u\|_\infty \int_{B_r(x)} \frac{\di \mu(y)}{V(y,r)} \le \|A_r^*1\|_\infty \|u\|_\infty.
    \end{align*}
Thus $A_r^* : L^\infty(X,\mu) \to L^\infty(X,\mu)$ is bounded with $\| A_r^* \|_{\infty \to \infty} \le \|A_r^*1\|_\infty $. The conclusion for $p\in (1,+\infty)$ follows from the Riesz-Thorin theorem and Lemma \ref{lem:L1}. 

Let us prove that $A_r$ and $A_r^*$ are adjoint of each other. Consider $u \in L^p(X,\mu)$ and $v\in L^q(X,\mu)$. Then
\begin{align*}
    \int_X A_r^*u(x)\, v(x) \di \mu(x) & = \int_X \int_X 1_{B_r(x)}(y)\,  \frac{u(y)}{V(y,r)} \, v(x) \di \mu(y) \di \mu(x)\\
    & = \int_X \frac{u(y)}{V(y,r)} \int_X 1_{B_r(x)}(y)\, \, v(x) \di \mu(x) \di \mu(y)\\
    & = \int_X u(y) A_rv(y) \di \mu(y)
\end{align*}
where we used the Fubini--Tonelli to get the second equality, and the equality $1_{B_r(x)}(y) = 1_{B_r(y)}(x)$ to get the last one.  As for the compactness of $A_r^*$ under \eqref{eq:V-1_int}, this result is a direct consequence of the Schauder theorem for compact operators which can be applied thanks to Lemma \ref{lem:average}.
\end{proof}

\subsection{Discussion on the assumptions}

Let us discuss the validity of \eqref{I} and  \eqref{eq:V-1_int}.   Recall first that : \eqref{eq:V-1_int} $ \Rightarrow $ \eqref{I}. Both properties hold on totally bounded spaces, as seen in the next lemma.

\begin{lemma}\label{lem:compact_space}
Assume that $(X,\dist)$ is totally bounded. Then \eqref{eq:V-1_int} (and then \eqref{I}) holds for any $r>0$.  
\end{lemma}

\begin{proof}
      Consider $r>0$ and a finite cover $\{B_{r/2}(x_i)\}$ of $X$. For any $x \in X$, there exists $i$ such that $x \in B_{r/2}(x_i)$. Then $B_r(x)$ contains $B_{r/2}(x_i)$ so $V(x,r) \ge V(x_i,r/2) \ge \min_j V(x_j,r/2)>0$. Thus
      \[
      \int_X \frac{\di \mu(x)}{V(x,r)} \le \frac{\mu(X)}{\min_j V(x_j,r/2)} < +\infty.
      \]
\end{proof}

\begin{rem}\label{rem:Ar}
If $\mu(X)=+\infty$ and $M(r)<+\infty$ then \eqref{eq:V-1_int} cannot hold :
\[
\int_X \frac{\di \mu(x)}{V(x,r)} \ge \frac{\mu(X)}{M(r)} = +\infty.
\]
This happens, for instance, on $\mathbb{R}^n$ endowed with the Euclidean distance and the Lebesgue measure. More generally, this property cannot hold on a locally compact topological group endowed with a left-invariant metric compatible with the Haar measure and with infinite volume (see \cite[Lemma 1]{Struble} for more details about these spaces). 

If $\mu(X)<\infty$ and $m(r)>0$, then \eqref{eq:V-1_int} always holds :
\[
\int_X \frac{\di \mu(x)}{V(x,r)} \le \frac{\mu(X)}{m(r)} < +\infty.
\]
In this regard, observe that if $X$ is not totally bounded, then there exist $r>0$ small enough and a countable family of disjoint balls $\{B_{r}(x_i)\}$ in $X$, so that $\mu(X)<\infty$ and $m(r)>0$ cannot hold simultaneously:
\[
\mu(X) \ge \sum_{i} V(x_i,r).
\]
\end{rem}

Let us now focus on \eqref{I}. We show below that this condition holds on so-called doubling spaces.  Let us recall this classical property and its uniform local variant, see e.g.~\cite{HKST} for more details.
 
 \begin{definition}\label{def:loc_doubling}
The space $(X,\dist,\mu)$ is called globally doubling if there exists $C>0$ such that for any $x \in X$ and $r>0$,
\begin{equation}\label{eq:rloc_doubling}
 V(x,2r) \le CV(x,r).
 \end{equation}
It is called uniformly locally doubling if there exist $C,r_0>0$ such that \eqref{eq:rloc_doubling} holds for any $x \in X$ and $r \in (0,r_0)$.
\end{definition}

The celebrated Bishop--Gromov theorem (see e.g.~\cite[Theorem III.4.5]{Chavel}) implies that any complete Riemannian manifold with a uniform lower bound on the Ricci curvature is uniformly uniformly locally doubling, and globally doubling if the uniform bound is non-negative. This is also true for metric spaces with generalized sectional curvature bounded from below in the sense of Alexandrov \cite[Theorem 10.6.6]{BBI} and $\mathrm{CD}(K,N)$ metric measure space \cite[Corollary 30.14]{Villani}.

The next lemma relates the uniformly local doubling condition with \eqref{I}.

\begin{lemma}
Let $(X,\dist,\mu)$ be uniformly locally doubling with parameters $C,r_0$. Then \eqref{I} holds with $\|A_r^*1\|_\infty \le C$ for any $r \in (0,r_0)$.
\end{lemma}

\begin{proof}
For any $x \in X$ and $r \in (0,r_0)$, the triangle inequality yields that $B_r(x) \subset B_{2r}(y)$ for any $y \in B_r(x)$. Then
\[
A_r^*1(x) = \fint_{B_r(x)} \frac{V(x,r)}{V(y,r)} \di \mu(y) \le \fint_{B_{r}(x)} \frac{V(y,2r)}{V(y,r)} \di \mu(y) \le C.
\]
\end{proof}

\begin{rem}
Of course if $(X,\dist,\mu)$ is globally doubling with constant $C$, then \eqref{I} holds with $\|A_r^*1\|_\infty \le C$ for any $r >0$.
\end{rem}

\begin{rem}
The previous result notably implies that \eqref{I} may hold in situations where \eqref{eq:V-1_int} does not. This happens e.g.~on a non-compact Riemannian manifold $(M,g)$ with non-negative Ricci curvature endowed with its canonical Riemannian distance $\dist$ and volume measure $\mu$. Indeed, such a space has infinite volume, and the Bishop--Gromov theorem implies that $M(r) \le \mathbb{V}^n(r)$ for any $r>0$, where $\mathbb{V}^n(r)$ is the Lebesgue measure of an Euclidean ball of radius $r$ in $\mathbb{R}^n$. From Remark \ref{rem:Ar}, we get that $M$ cannot satisfy \eqref{eq:V-1_int} for any $r>0$, while it does satisfies \eqref{I} thanks to the global doubling condition.
\end{rem}

We conclude this discussion with two final remarks.  First, if $m(r)>0$ and $M(r)<+\infty$, then \eqref{I} always holds with
    \[
    \|A_r^*1\|_\infty \le \frac{M(r)}{m(r)}\, \cdot 
    \]
    This happens on locally compact topological groups endowed with a left-invariant distance and their Haar measure, compare with Remark \ref{rem:Ar}.  Secondly, \eqref{I} can easily be seen as a weak variant of the comparability conditions introduced in \cite{Aldaz1,MT2}.

\subsection{Symmetrization}

For any $x,y \in X$ and $r>0$, set
\begin{align*}
\tilde{a}_r(x,y) & = \frac{1}{2} \left( a_r(x,y) + a^*_r(x,y) \right)\\
& = \frac{1}{2} \left( \frac{1}{V(x,r)} + \frac{1}{V(y,r)} \right) 1_{B_r(x)} (y)
\end{align*}
Consider $u \in L^1_{\text{loc}}(X,\mu)$ such that $v(\cdot) \df u(\cdot)/V(\cdot,r) \in L^1_{\text{loc}}(X,\mu)$. For any $x \in X$ and $r>0$ such that $u$ and $v$ are $\mu$-integrable on $B_r(x)$, set
\begin{equation}\label{eq:symm}
\tilde{A}_ru(x) \df \frac{1}{2}(A_ru(x) + A_r^*u(x)) = \int_{X} \tilde{a}_r(x,y) u(y)\di \mu(y).
\end{equation}
Then the next lemma is an obvious consequence of Lemma \ref{lem:average} and Lemma \ref{lem:adjoint}.

\begin{corollary}\label{lem:sym}
    Assume that there exists $r>0$ such that \eqref{I} holds. Then $\tilde{A}_r : L^2(X,\mu) \to L^2(X,\mu)$ is a self-adjoint operator such that
    \[
    \|\tilde{A}_r\|_{2 \to 2} \le \|A_r^*1\|_\infty^{1/2}.
    \]
    Moreover, if \eqref{eq:V-1_int} holds, then $\tilde{A}_r : L^2(X,\mu) \to L^2(X,\mu)$ is compact.
\end{corollary}

\section{Symmetrized AMV operators}

In this section, we provide our working definition of the symmetrized AMV $r$-Laplace operator
$\tilde{\Delta}_r$ and we derive several spectral properties in a general setting.

\subsection{Definitions}

For this subsection, we consider a metric measure space $(X,\dist,\mu)$ satisfying \eqref{I} for some fixed $r>0$. 

\begin{definition}
The symmetrized AMV $r$-Laplace operator of $(X,\dist,\mu)$ is
\[
\tilde{\Delta}_r \df \frac{1}{r^2} (\tilde{A}_r - [\tilde{A}_r1] \mathrm{I})
\]
where we recall that $\tilde{A}_r$ is defined in \eqref{eq:symm}.
\end{definition}

\begin{rem}
We may use the notation $\tilde{\Delta}_{r,\mathfrak{X}}$ to specify that we work on the metric measure space $\mathfrak{X} = (X,\dist,\mu)$.
\end{rem}

\begin{lemma}\label{lem:sAMV}
$\tilde{\Delta}_r$ is a bounded, self-adjoint operator acting on $L^2(X,\mu)$ with
\begin{equation}\label{eq:boundAMV}
\|\tilde{\Delta}_r\|_{2 \to 2} \le \frac{1}{2r^2} \left( 2 \|A_r^*1\|_\infty^{1/2}+ \|A_r^*1\|_\infty + 1 \right).
\end{equation}
\end{lemma}

\begin{proof}
The self-adjointness of $\tilde{\Delta}_r$ is obvious because $\tilde{A}_r$ and $[\tilde{A}_r1] \mathrm{I}$ are self-adjoint too. The boundedness is a consequence of Lemma \ref{lem:sym}. Indeed,
\begin{align*}
    \|\tilde{\Delta}_r\|_{2 \to 2} & \le \frac{1}{r^2} \left( \|\tilde{A}_r\|_{2 \to 2} + \|[\tilde{A}_r 1 ]I\|_{2 \to 2}\right) \\
    & \le \frac{1}{r^2} \left(\|A_r^*1\|_\infty^{1/2}+ \|\tilde{A}_r 1 \|_{\infty}\underbrace{\|I\|_{2 \to 2}}_{=1}\right) \\
    & \le \frac{1}{r^2} \left(\|A_r^*1\|_\infty^{1/2} + \frac{\|A_r 1 \|_{\infty} + \|A_r^* 1 \|_{\infty}}{2}\right)\\
    & = \frac{1}{r^2} \left(\|A_r^*1\|_\infty^{1/2} + \frac{1 + \|A_r^*1\|_\infty}{2}\right)
\end{align*}
hence $\tilde{\Delta}_r: L^2(X,\mu) \to L^2(X,\mu)$ is bounded and \eqref{eq:boundAMV} holds. 
\end{proof}

\begin{definition}
The energy functional $\tilde{E}_r$ of $(X,\dist,\mu)$ is the quadratic form on $L^2(X,\mu)$ defined by
\[
\tilde{E}_r(f) \df \frac{1}{4} \int_X \int_X 1_{B_r(x)}(y)\left( \frac{1}{V(x,r)} +\frac{1}{V(y,r)} \right)  \left( \frac{f(x) -f(y)}{r} \right)^2 \di \mu(y) \di \mu(x).
\]
The associated bilinear form, which we still denote by $\tilde{E}_r$, is given by
    \begin{align*}
\tilde{E}_r(f,\psi) & \df \frac{1}{4} \int_X \int_{X} 1_{B_r(x)}(y)\left( \frac{1}{V(x,r)} +\frac{1}{V(y,r)} \right) \\  
& \qquad \qquad \qquad \qquad \times \frac{(f(x) -f(y))(\psi(x)-\psi(y))}{r^2} \di \mu(y) \di \mu(x).
    \end{align*}
\end{definition}

\begin{rem}
A suitable use of the Fubini--Tonelli theorem shows that the energy functional $\tilde{E}_r(f)$ equals the approximate Korevaar--Schoen energy \cite{KS}
\[
\frac{1}{2} \int_X \fint_{B_r(x)} \frac{|f(y)-f(x)|^2}{r^2} \di \mu(y) \di \mu(x).
\]
\end{rem}

\begin{rem}
We may also use the notation $\tilde{E}_{r,\mathfrak{X}}$ to specify the metric measure space $\mathfrak{X} = (X,\dist,\mu)$.
\end{rem}

The next lemma goes back to \cite[Lemma 3.1]{AKS2}. We provide a quick proof for completeness.

    \begin{lemma}
For any $f,\psi \in L^2(X,\mu)$,
\begin{equation}\label{eq:IBP}
    \tilde{E}_r(f,\psi) = \langle - \tilde{\Delta}_r f, \psi \rangle_{L^2}.
\end{equation}
\end{lemma}

\begin{proof}Note that
\begin{align*}
    \tilde{E}_r(f,\psi) & = \frac{1}{4} \int_X \int_{B_r(x)}\left( \frac{1}{V(x,r)} +\frac{1}{V(y,r)} \right) \frac{(f(x) -f(y))\psi(x)}{r^2} \di \mu(y) \di \mu(x) \\
    & - \frac{1}{4} \int_X \int_{X} 1_{B_r(x)}(y)\left( \frac{1}{V(x,r)} +\frac{1}{V(y,r)} \right) \frac{(f(x) -f(y))\psi(y)}{r^2} \di \mu(y) \di \mu(x).
\end{align*}
Using $1_{B_r(x)}(y) = 1_{B_r(y)}(x)$ and then the Fubini theorem, we can rewrite the second term as the opposite of the first one, so that we eventually get \eqref{eq:IBP}.
\end{proof}

\begin{rem}\label{rem:non-neg}
Observe that \eqref{eq:IBP} implies that $- \tilde{\Delta}_r$ is a non-negative operator, since for any $f \in L^2(X,\mu)$,
\[
\langle - \tilde{\Delta}_r f, f \rangle_{L^2} = \tilde{E}_r(f) \ge 0
\]
\end{rem}

Let us recall the definition of spectrum.

\begin{definition}
We let $\sigma(-\tilde{\Delta}_r)$ denote the spectrum of $-\tilde{\Delta}_r$, that is to say, the set of elements $\lambda \in \bC^*$ such that $-\tilde{\Delta}_r - \lambda \mathrm{I} : L^2(X,\mu) \to L^2(X,\mu)$ is not a bijection.
\end{definition}
It is well-known from classical functional analysis that the spectrum $\sigma(T)$ of a bounded operator $T$ acting on a Banach space $E$ can be decomposed as
\[
\sigma(T) = \sigma_{\text{p}}(T) \cup \sigma_{ \text{c}}(T) \cup \sigma_{ \text{a}}(T)
\]
where:
\begin{itemize}
\item $\sigma_{\text{p}}(T)$ is the point spectrum, that is to say, the set of $\lambda \in \bC^+$  such that $(T - \lambda \mathrm{I})f = 0$ for some non-zero $f \in E$, in which case $\lambda$ is called an eigenvalue and $f$ an eigenvector of $T$,
\item $\sigma_{ \text{c}}(T)$ is the compression spectrum, that is to say, the set of $\lambda \in \bC^+$  whose conjugate $\bar{\lambda}$ is an eigenvalue of the adjoint $T^*$,
\item $\sigma_{\text{a}}(T)$ is the approximate point spectrum, that is to say, the set of $\lambda \in \bC^+$ for which there exists $(f_n) \subset E$ with $\|f_n\| = 1$ for any $n$ such that $\|(T - \lambda \mathrm{I})f_n\| \to 0$.
\end{itemize}
Since $-\tilde{\Delta}_r$ is self-adjoint and non-negative, we know that
\[
\sigma(-\tilde{\Delta}_r) \subset [0, +\infty].
\]
This implies that $\sigma_{\text{p}}(-\tilde{\Delta}_r) = \sigma_{\text{c}}(-\tilde{\Delta}_r)$, so that
\begin{equation}\label{eq:decomp}
\sigma(-\tilde{\Delta}_r) = \sigma_{\text{p}}(-\tilde{\Delta}_r) \cup \sigma_{\text{a}}(-\tilde{\Delta}_r).
\end{equation}

\begin{definition}\label{def:min-max}
For any $k \in \mathbb{N}$, we define
\begin{equation}\label{eq:min-max}
\tilde{\lambda}_{k,r} \df   \inf_{V \in \mathcal{G}_{k+1}(L^2(X,\mu))}
    \sup_{f \in V}
    \frac{\tilde{E}_r(f)}
    {\|f\|_2},
\end{equation}
where $\mathcal{G}_{k+1}(L^2(X,\mu))$ is the $(k+1)$-th Grassmannian of $L^2(X,\mu)$.
\end{definition}

\begin{rem}\label{rem:Courant_etc} Let $\sigma_{\text{ess}}(-\tilde{\Delta}_r)$ denote the essential spectrum of $-\tilde{\Delta}_r$, i.e.~the closed subset of $\sigma(-\tilde{\Delta}_r)$ made of those $\lambda$ such that $-\tilde{\Delta}_r - \lambda \mathrm{I}$ is not a Fredholm operator. Since $-\tilde{\Delta}_r$ is self-adjoint, the Fischer--Polyà minimum-maximum principle (see e.g.~\cite[p.12]{WeinsteinStenger}) asserts that if there exists a positive integer $N$ such that $\tilde{\lambda}_{N,r} < \min  \sigma_{\text{ess}}(-\tilde{\Delta}_r)$ then $-\tilde{\Delta}_r$ admits $N+1$ isolated eigenvalues
$$\lambda_0(-\tilde{\Delta}_r) \le \ldots \le \lambda_N(-\tilde{\Delta}_r) < \min  \sigma_{\text{ess}}(-\tilde{\Delta}_r)$$
such that for any $k \in \{0,\ldots,N\}$,
$$
\tilde{\lambda}_{k,r}  = \lambda_{k}(-\tilde{\Delta}_r).
$$
\end{rem}

\subsection{Spectral properties}

Our first spectral result on $\tilde{\Delta}_r$ is the following. Note that we need the compactness of $\tilde{A}_r$ here, thus we assume \eqref{eq:V-1_int}.

\begin{prop}\label{prop:eigenvalueProperties}
   Let $(X,\dist,\mu)$ be a metric measure space satisfying \eqref{eq:V-1_int} for some fixed $r>0$. Assume that $\lambda \in \sigma(-\tilde{\Delta}_r)$ satisfies
    \begin{equation}\label{eq:cond_on_lambda}
        \lambda < \inf_{y \in X}[\tilde{A}_r1](y)/{r^2}.
    \end{equation}
    Then $\lambda$ is an isolated eigenvalue of $-\tilde{\Delta}_r$ which does not belong to $\sigma_{\text{ess}}(-\tilde{\Delta}_r)$.
\end{prop}

\begin{proof}
 Let us first show that $\lambda$ is an eigenvalue, that is to say, that $\lambda$ belongs to the point spectrum. According to \eqref{eq:decomp}, it is enough to show that if $\lambda$ is in the approximate point spectrum, then it is in the point spectrum. If this is the case, then there exists a sequence $(f_n) \in L^2(X,\mu)$ such that $\|f_n\|_{L^2(X)} = 1$ and
    \begin{equation}
    \label{approximateSpectrum}
        \|-\tilde{\Delta}_rf_n - \lambda f_n\|_{L^2(X)}
    =
        \left\|
        \left(
            \frac{[\tilde{A}_r1]}{r^2}
        -
            \lambda
        \right)f_n
        -
        \frac{\tilde{A}_rf_n}{r^2}
        \right\|_{L^2(X)}
    \to
        0.
    \end{equation}
    Since $\tilde{A}_r$ is compact, we have that $\tilde{A}_rf_n$ converges up to a subsequence, and as such by equation \eqref{approximateSpectrum} so does $\left(
            [\tilde{A}_r1]
        -
            r^2\lambda
        \right)f_n$, with limit $g \in L^2(X)$. Consider $\delta>0$ such that $0 \leq \lambda + \delta/r^2 < \inf_{y \in X}[\tilde{A}_r1](y)/{r^2}$ and define
        \begin{equation}
            b_r(x) := [\tilde{A}_r1]-r^2\lambda \geq \delta.
        \end{equation}
        Thus we have that $f:=g/b_r \in L^2(X,\mu)$, and so
        \begin{equation}
            \delta \|f_n - \frac{g}{b_r}\|_{L^2(X)}
        \leq
            \|b_r f_n - g \|_{L^2(X)}
        \rightarrow 0.
        \end{equation}
        Thus $f_n$ converges in $L^2(X,\mu)$ to the limit function $f$. Using continuity of $-\tilde{\Delta}_r$ we conclude that
        \begin{equation}
            -\tilde{\Delta}_r f
        =
            \lambda f,
        \end{equation}
        and thus $\lambda$ is in the point spectrum.

        To prove that $\lambda$ is an isolated point, we suppose by contradiction that there exists an infinite sequence $(\lambda_n) \subset \sigma(-\tilde{\Delta}_r)$ of distinct values such that $\lambda_n \rightarrow \lambda$. Then there exists $\delta>0$ such that for any high enough $n$,
        \begin{equation}
            \lambda_n + \delta/r^2 < \inf_{y \in X}[\tilde{A}_r1](y)/{r^2}, 
            \quad \quad
            \lambda_n \rightarrow \lambda.
        \end{equation}
        From the previous paragraph, we know that $\lambda_n$ is in the point spectrum, thus there exists $f_n \in L^2(X, \mu)$ satisfying $\|f_n\|_{L^2(X)} = 1$, such that
        \begin{equation}
            -\tilde{\Delta}_r
            f_n
        =
            \lambda_n
            f_n.
        \end{equation}
        This can be written as 
        \begin{equation}
        \label{eqForDiscreteSpectrum}
            -\frac{\tilde{A}_r}{r^2}f_n
        =
            (\lambda_n-\frac{[\tilde{A}_r1]}{r^2})f_n.
        \end{equation}
        Using compactness of $\tilde{A}_r$, we know that $\tilde{A}_rf_n$ converges up to a subsequence. This implies that $ (r^2\lambda_n-[\tilde{A}_r1])f_n$ converges up to a subsequence to some $g \in L^2(X,\mu)$. Define
        \begin{equation}
            b_{r,n}(x)
        :=
            [\tilde{A}_r1]-r^2\lambda_n.
        \end{equation}
        With this we have that
        \begin{align}
            \delta\left\|f_n - \frac{g}{b_r}\right\|_{L^2(X)}
        &\leq
            \delta\left(
            \left\|f_n - \frac{g}{b_{r,n}}\right\|_{L^2(X)}
        +
            \left\|\frac{g}{b_{r,n}}-\frac{g}{b_{r}}\right\|_{L^2(X)}
            \right)\\
        &\leq
        \|b_{r,n}f_n - g\|_{L^2(X)}
        +
        \delta\left\|\frac{g}{b_{r,n}}-\frac{g}{b_{r}}\right\|_{L^2(X)}.
        \end{align}
        We have that $\|b_{r,n}f_n - g\|_{L^2(X)} \rightarrow 0$ and also $\|\frac{g}{b_{r,n}}-\frac{g}{b_{r}}\|_{L^2(X)} \rightarrow 0$ since $0< \delta \leq b_{r,n},b_r$ and $\lambda_n \rightarrow \lambda$. Thus $f_n$ converges. However since all the eigenvalues are different, we know that $\langle f_n, f_j \rangle = \delta_{n,j}$, and so the sequence cannot converge up to a subsequence, achieving contradiction. This shows that $\lambda$ is an isolated point of $\sigma(-\tilde{\Delta}_r)$ finishing the first part of the proof.

        Let us now prove that $-\tilde{\Delta}_r - \lambda I$ is a Fredholm operator.
        
        To show that $\ker(-\tilde{\Delta}_r - \lambda I)$ is finite dimensional we proceed by contradiction. Assume that there exists an infinite sequence $(f_n) \subset \ker(-\tilde{\Delta}_r - \lambda I)$ such that $\langle f_n, f_j \rangle = \delta_{n,j}$. Thus
        \begin{equation}
            -\frac{\tilde{A}_rf_n}{r^2}
        =
            \left(
                \lambda
            -
                \frac{[\tilde{A}_r 1]}{r^2}
            \right)f_n.
        \end{equation}
        Similar to before we can use compactness of $\tilde{A}_r$ and the condition on $\lambda$ to conclude that $f_n$ converges in $L^2(X, \mu)$ up to a subsequence. However, this is prevented by $\langle f_n, f_j \rangle = \delta_{n,j}$.

        To show that the image of $-\tilde{\Delta}_r - \lambda I$ is closed, consider a sequence $g_n:= (-\tilde{\Delta}_r - \lambda I)(f_n)$ such that $g_n \rightarrow g$. Similarly to before, we can conclude that since $g_n$ converges, then $f_n$ converges to some $f$, and thus $g = (-\tilde{\Delta}_r - \lambda I)(f)$.
\end{proof}

\begin{corollary}\label{cor:min}
Let $(X,\dist,\mu)$ be a metric measure space such that for some $r_0>0$ the assumption \eqref{eq:V-1_int} holds for any $r \in (0,r_0)$. Then
    \[
    \lim\limits_{r \downarrow 0} \big( \min \sigma_{\text{ess}}(-\tilde{\Delta}_r) \big) = +\infty.
    \]
\end{corollary}

\begin{proof}
Proposition \ref{prop:eigenvalueProperties} implies that
    \begin{equation}\label{eq:essential_spectrum_bound}
        \inf_{y \in X}[\tilde{A}_r1](y)/{r^2} \le \min  \sigma_{\text{ess}}(-\tilde{\Delta}_r).
    \end{equation}
But for any $r>0$,
\[
\tilde{A}_r1 = \frac{1}{2} \left( A_r1 + A^*_r1\right) \ge \frac{1}{2} A_r1 = \frac{1}{2}
\]
hence \eqref{eq:essential_spectrum_bound} implies that
    \[
    \min \sigma_{\text{ess}}(-\tilde{\Delta}_r) \ge \frac{1}{2r^2} \stackrel{r \downarrow 0}{\to} +\infty.
    \]
\end{proof}

Let us provide our second spectral result on $\tilde{\Delta}_r$.

\begin{prop}\label{prop:kernelOfSymAMV}
 Let $(X,\dist,\mu)$ be a connected metric measure space satisfying \eqref{eq:V-1_int} for some fixed $r>0$. 
  Then the kernel of $\tilde{\Delta}_r$ contains constant functions only, and $\tilde{E}_r$ defines a scalar product on
     \begin{equation}
    \label{eq:perp}
      \mathrm{\Pi}(X,\mu) \df \left\{ f \in L^2(X,\mu) : \int_X f \di \mu = 0\right\}.
    \end{equation}
\end{prop}

\begin{proof}
Consider $f \in L^2(X,\mu) \backslash \{0\}$ such that $\tilde{\Delta}_r f = 0$. Then $\tilde{E}_r(f) = 0$. This implies that for $\mu$-a.e. $x \in X$,
    \begin{equation}
        \int_X
        1_{B_r(x)}(y)
        \left(
            \frac{1}{V(x,r)}
            +
            \frac{1}{V(y,r)}
        \right)
        \left(
            \frac{f(x)-f(y)}{r}
        \right)^2
        d\mu(y)
    =
        0
    \end{equation}
    which implies, in turn,
    \begin{equation}
    \label{ConstantInBalls}
        \mu
        \left(
            \{
                y \in B_r(x):
                f(y)
                =
                f(x)
            \}
        \right)
    =
        \mu(B_r(x)).
    \end{equation} 
    Consider $F \df \{x \in X :  \text{$f(x)$ is a well-defined real number}\}$ and 
    \begin{equation}
        A
    :=
        \{
            x \in F: 
            \mu
        \left(
            \{
                y \in B_r(x):
                f(y)
                =
                f(x)
            \}
        \right)
    =
        \mu(B_r(x))
        \}.
    \end{equation}
    Then
    \begin{equation}
        \mu(X \backslash A) = 0.
    \end{equation}
    Take $z \in A$ and let $c = f(z)$. Consider
    \begin{equation}
        I = \{x \in A : f(x) = c\}, \quad\quad I' = \{x \in A : f(x) \neq c\},
    \end{equation}
    and notice that $I \cup I' = A$. Suppose by contradiction that $I' \neq \emptyset$.
    Set
    \begin{equation}
        W := \bigcup_{x \in I} B_r(x), \quad \quad
        V := \bigcup_{y \in I'} B_r(y).
    \end{equation}
    Since $V$ and $W$ are open sets whose union contains $A$ which has full measure in $X$, we must have
    \begin{equation}
        W \cup V = X,
    \end{equation}
    otherwise $X \backslash A$ would contain an open ball with positive measure. Since $W$ and $V$ form an open cover of $X$, and $X$ is connected, if both $V$ and $W$ are different from the empty set, then there exist $x \in I$ and $y \in I'$ such that
    \begin{equation}
        B_r(x) \cap B_r(y) \neq \emptyset.
    \end{equation}
    However
    \begin{align}
        f|_{B_r(x) \cap B_r(y)}(w) &= f(x) \quad\quad
        \mu\text{-a.e. } w \in X,\\
        f|_{B_r(x) \cap B_r(y)}(w) &= f(y) \quad\quad
        \mu\text{-a.e. } w \in X.
    \end{align}
    This is not possible since $f(x) \neq f(y)$ and $\mu(B_r(x) \cap B_r(y)) > 0$. This implies that $I = A$ and that $f(w) = c$ for $\mu$-a.e. $w \in X$.
    Then $-\tilde{\Delta}_r$ has a non-trivial kernel consisting of the constant functions only. Moreover, since $-\tilde{\Delta}_r$ is non-negative (Remark \ref{rem:non-neg}), we get the desired property on \eqref{eq:perp}.
\end{proof}

We are now in a position to prove Theorem \ref{th:1}. We recall that the context of this statement is a compact uniformly locally doubling metric measure space $(X,\dist,\mu)$. The compactness of the space ensures that \eqref{eq:V-1_int} holds for any $r>0$, see Lemma \ref{lem:compact_space}.

\begin{proof}
For $k \ge 1$ integer,  let $x_0,\ldots,x_k \in X$ be distinct points. Set
\[
\bar{r}_k \df\min_{0\le i \neq j \le k} \frac{\dist(x_i,x_j)}{4} \, \cdot
\]
For any $i \in \{0,\ldots,k\}$ and $y \in X$, define
\[
 \tilde{f}_i (y) \df \left( 1 - \frac{\dist(x_i,y)}{\bar{r}_k }\right)^+ \qquad \text{and} \qquad f_i(y) \df \frac{\tilde{f}_i(y)}{\|\tilde{f}_i\|_{2}}\, \cdot
\]
Note that each $f_i$ is an $L^2$-normalized Lipschitz function supported in $B_{\bar{r}_k}(x_i)$, and that $(f_0,\ldots,f_k)$ is an orthonormal family of $L^2(X,\mu)$. Set
\[
V \df \mathrm{Span}(f_0,\ldots,f_k) \in \mathcal{G}_{k+1}(L^2(X,\mu))
\]
and observe that for any $r>0$,
\[
\tilde{\lambda}_{k,r} \le \max_{\substack{f \in V\\\|f\|_2 = 1}} \tilde{E}_r(f).
\]
Consider $r \in (0,\bar{r}_k)$ and $i \neq j$ in $\{0,\ldots,k\}$. If $x \in B_{2 \bar{r}_k}(x_i)$, then $f_j(y)=0$ for any $y \in B_r(x)$, while if $x \notin B_{2 \bar{r}_k}(x_i)$, then $f_i(y)=0$ for any $y \in B_r(x)$. In both cases,
\[
(f_i(x)-f_i(y))(f_j(x)-f_j(y)) = 0
\]
for any $y \in B_r(x)$. Thus
\begin{equation}\label{eq:null}
\tilde{E}_r(f_i,f_j) = 0.
\end{equation}
Consider $f \in V$ such that $\|f\|_2=1$. Then $f=\sum_{i=0}^k a_i f_i$ for some $a_0,\ldots, a_k \in \mathbb{R}$ such that $\sum_{i=0}^k a_i^2 = 1$. By \eqref{eq:null}, we get
\[
\tilde{E}_r(f)  = \sum_{i=0}^k a_i^2 \tilde{E}_r(f_i) \le \max_{0 \le i \le k} \tilde{E}_r(f_i).  
\]
Let $r_0,C$ be the parameters of the uniform local doubling property of $(X,\dist,\mu)$, see Definition \ref{def:loc_doubling}. Then for any $i \in \{0,\ldots,k\}$ and $r \in (0, r_0)$,
\begin{align*}
\tilde{E}_r(f_i) & \le \frac{\Lip^{ 2 }(f_i)}{4} \int_X \fint_{B_r(x)} \underbrace{\left( 1 + \frac{V(x,r)}{V(y,r)} \right)}_{\le 1+ C^2} \underbrace{\frac{\dist^2(x,y)}{r^2} }_{\le 1} \di \mu(y)\di \mu(x)\\
& \le \frac{\Lip^{{ 2} }(f_i)(1+C^2)\mu(X)}{4} \, \cdot
\end{align*}
Therefore, for any $r < \tilde{r}_k \df \min(\bar{r}_k,R)$, we get
\begin{equation}
\label{eq:boundOfEigenvalues}
\tilde{\lambda}_{k,r} \le \tilde{C} \df \frac{(1+C^2)\mu(X)}{4} \max_{0 \le i \le k} \Lip^{{ 2} }(f_i)  \, \cdot 
\end{equation}
By Corollary \ref{cor:min},  there exists $r_k \in (0,\tilde{r}_k)$ such that $\tilde{C}< \min \sigma_{\mathrm{ess}}(-\tilde{\Delta}_r)$ for any $r \in (0,r_k)$. Then Remark \ref{rem:Courant_etc} implies that for such an $r$ the operator $-\tilde{\Delta}_r$ admits $k+1$ eigenvalues $ \lambda_0(-\tilde{\Delta}_r) \le \lambda_1(-\tilde{\Delta}_r) \le  \ldots \le \lambda_k(-\tilde{\Delta}_r)$ such that $\lambda_i(-\tilde{\Delta}_r)=\tilde{\lambda}_{i,r}$ for any $i \in \{0,\ldots,k\}$. That $\lambda_0(-\tilde{\Delta}_r)=0$ follows from Proposition \ref{prop:kernelOfSymAMV}. Moreover, by Remark \ref{rem:Courant_etc} and Proposition \ref{prop:kernelOfSymAMV}, we know that 
\[
\tilde{\lambda}_{1,r} = \min_{f \in \Pi(X,\mu)} \frac{\tilde{E}_r(f)}{\|f\|_2}
\]
where $\Pi(X,\mu)$ is as in \eqref{eq:perp}. Since $-\Delta_r$ has a kernel which is $L^2$-orthogonal to $\Pi(X,\mu)$ (Proposition \ref{prop:kernelOfSymAMV}), we have $\tilde{E}_r(f)>0$ for any $f \in \Pi(X,\mu)$, hence we get
\[
\tilde{\lambda}_{1,r}
    >
        0.
\]
\end{proof}

\section{First eigenvalue of torus and hypercubes}\label{sec:torus_et_al}

In this section, we derive some results which will be applied in Section \ref{sec:manifolds}. We let $m$ be a positive integer kept fixed throughout the section.

\subsection{A preliminary lemma} We begin with a result where we use the normalized sinc function, namely
\[
\sinc(\rho) \df \begin{cases}\displaystyle \frac{\sin(\pi \rho)}{\pi \rho} \qquad &  \rho \in \mathbb{R} \backslash \{0\}, \\
1 & \rho = 0,
\end{cases}
\]
and the following notation: for any $p = (p_1,...,p_m) \in \mathbb{Z}^m$,
\begin{equation}\label{eq:notation}
J(p)  \df \{i\in \{1,\ldots,m\}: p_i \neq 0\}, \qquad j(p) \df \# J(p).
\end{equation}

\begin{lemma}
\label{SinLemma}
    \begin{equation}
    \label{eq:coeffsOfRectangleProp}
        \liminf_{r \rightarrow 0}
        \inf_{0 \neq p \in \mathbb{Z}^m}
        \left|
        \frac{1}{r^2}
        \left(
            1-
            \prod_{i \in J(p)}
            \sinc(p_i r)
        \right)
        \right|
    >
        0.
    \end{equation}
\end{lemma}
\begin{proof}
    We start by pointing out that for any $r > 0$ and $p \in \mathbb{Z}^m \backslash \{0\}$,
    \begin{equation*}
            \prod_{i \in J(p)}
            \sinc(p_i r)
        \neq
        1,
    \end{equation*}
    so that
    \begin{equation}
    \label{eq:ineqForFixedRandK}
        \left|
        \frac{1}{r^2}
        \left(
            1-
            \prod_{i \in J(p)}
            \sinc(p_i r)
        \right)
        \right|
    >
        0.
    \end{equation}
    Moreover, for any $r>0$, if  $|p|_\infty \df \max_{1 \le i \le m} |p_i| \to +\infty$, then
    \begin{equation*}
        \left|
        \frac{1}{r^2}
        \left(
            1-
            \prod_{i \in J(p)}
            \sinc(p_i r)
        \right)
        \right|
        \rightarrow
        \frac{1}{r^2}
        >
        0,
    \end{equation*}
    hence there exists $R>0$ such that 
    \begin{equation*}
        \inf_{0 \neq p \in \mathbb{Z}^m}
        \left|
        \frac{1}{r^2}
        \left(
            1-
            \prod_{i \in J(p)}
            \sinc(p_i r)
        \right)
        \right| =  \min_{\substack{0 \neq p \in \mathbb{Z}^m\\ |p|_\infty<R}}
        \left|
        \frac{1}{r^2}
        \left(
            1-
            \prod_{i \in J(p)}
            \sinc(p_i r)
        \right)
        \right|
    >
        0.
    \end{equation*}
    If \eqref{eq:coeffsOfRectangleProp} were to fail, due to the previous line, there would exist sequences $(r_n) \subset (0,+\infty)$ and $(p^{(n)}) \subset \mathbb{Z}^m \backslash \{0\}$ such that $r_n \to 0$ and
    \begin{equation}
    \label{eq:contradictHypSin}
        \lim_n\left|
        \frac{1}{r_n^2}
        \left(
            1-
            \prod_{i \in J(p)}
            \sinc(p^{(n)}_i r)
        \right)
        \right|
    =
        0.
    \end{equation}
    We have two cases.
    \begin{itemize}
        \item There exists $\alpha >0$ and $j \in \{1,\ldots,m\}$ such that $\liminf_n p^{(n)}_j r_n > \alpha$.
        \item For any $j \in \{1,\ldots,m\}$, one has $\liminf_n p^{(n)}_j r_n = 0$.
    \end{itemize}
 If the first one were true, then we would have
    \begin{equation}
        \liminf_n
        \left|
            1-\prod_{i \in J(p)}
            \sinc(p^{(n)}_i r)
        \right|
    >
        |
            1-\alpha
        |,
    \end{equation}
    and so \eqref{eq:contradictHypSin} couldn't hold.
    On the contrary, if the second case were true, up to extracting a subsequence we would have $p^{(n)}_j r_n \to 0$ for any $j \in \{1,\ldots,m\}$. For any $y = (y_1,\ldots,y_m) \in \bR^m$, set
        \begin{equation}
            G(y)
        \df
        1
        -
            \prod_{j = 1}^{m}
            \sinc(y_j).  
        \end{equation}
        Then $G$ is smooth on $\bR^m$ and satisfies
        \begin{equation}
            G(y)
        =
           \frac{1}{2} |y|^2
        +
            o(|y|^3), \qquad |y| \to 0.
        \end{equation}
        As a consequence, for $y \in \bR^m$ such that $|y|$ is small enough,
        \begin{equation}
            \left|
                G(y)
            \right|
        \geq
            \frac{1}{4} |y|^2.
        \end{equation}
        Then, for large enough $n$,

        \begin{align*}
            \left| 1- \prod_{i \in I(p^{(n)})} \sinc(p^{(n)}_i r_n)\right| & = G(r_n p^{(n)}) \geq \frac{1}{4} |r_n p^{(n)}|^2 \geq \frac{1}{4}r_n^2,
        \end{align*}
        because $p^{(n)} \neq 0$ implies that there exists at least one $i$ such that $|p^{(n)}_i|\geq 1$.
        Thus we obtain
        \begin{equation}
            \lim_n\left|
        \frac{1}{r_n^2}
        \left(
            1-
            \prod_{i \in J(p)}
            \sinc(\pi p^{(n)}_i r_n)
        \right)
        \right|
        \geq
        \frac{1}{4}
        \end{equation}
          and so \eqref{eq:contradictHypSin} could not hold. This concludes the proof.
\end{proof}

\subsection{Torus} Consider the torus
\begin{equation*}\label{eq:TorusDef}
    \mathbb{T}^m
:=
    \mathbb{R}^m/\big( -1 + 2\mathbb{Z}\big)^m
\end{equation*}
with its natural quotient map $\pi: \mathbb{R}^m \to \mathbb{T}^m$. Let $\pi^{-1}$ be the inverse of the bijective map
\[
\pi : [-1,1)^m \to \mathbb{T}^m.
\]
Let $\dist_\infty$ be the distance in $\mathbb{R}^m$ associated with the infinity norm, given by
\begin{equation}
    \dist_\infty(x,y)
:=
    \max\{|x_i-y_i|:i \in \{1,...,m\}\}
\end{equation}
for any $x=(x_1,\ldots,x_m)$ and $y = (y_1,\ldots,y_m)$ in $\mathbb{R}^m$. With respect to this distance, the open ball of radius $r>0$ centered at $x \in \mathbb{R}^m$ is 
\begin{equation}
\label{CubicBalls}
    Q_r(x)
:=
    \prod_{i=1}^m (-r + x_i, x_i + r).
\end{equation}
For any  $x,y \in \mathbb{T}^m$, set
\begin{equation*}
    \tilde{\dist}_\infty(x,y)
=
    \inf_{z \in \pi^{-1}(x), w \in \pi^{-1}(y)}
    \dist_\infty(z,w).
\end{equation*}
Then $\tilde{\dist}_\infty$ defines a distance on $\mathbb{T}^m$, and we denote by $\tilde{Q}_r(x)$ the open ball of radius $r>0$ centered at $x \in \mathbb{T}^m$ with respect to this distance. We also introduce the probability measure
\[
\mathbb{L}^m \df \pi_{\#}\left(  \frac{\mathcal{L}^m}{2^m} \right)
\]
on $\mathbb{T}^m$. Note that this is also the normalized Haar measure of $\mathbb{T}^m$ seen as a Lie group.

It is obvious that the metric measure space $\mathfrak{T}^m \df (\mathbb{T}^m, \tilde{\dist}_\infty, \mathbb{L}^m)$ satisfies the assumptions of Theorem \ref{th:1}, hence we know that for any small enough $r$,
$$\tilde{\lambda}_{1,r}=\lambda_{1}(- \tilde{\Delta}_{r})>0.$$
Then the following holds.

\begin{prop}\label{prop:torus}
    \begin{equation}
        \liminf_{r \rightarrow 0}
        \lambda_{1}(- \tilde{\Delta}_{r}) > 0.
    \end{equation}
\end{prop}

\begin{proof}
For any $x \in \mathbb{T}^m$ and $r>0$ small enough, the ball $\tilde{Q}_r(x) \subset \mathbb{T}^m$ is given by 
\begin{equation}
    \tilde{Q}_r(x)
=
    \pi(Q_r(\tilde{x}))
\end{equation}
where $\tilde{x}$ is any element in $\pi^{-1}(x)$, and $Q_r(\tilde{x})$ is as in \eqref{CubicBalls}. Then
\begin{equation}\label{eq:meas}
\mathbb{L}^m(\tilde{Q}_r(x))  = \frac{\mathcal{L}^m (Q_r(\tilde{x}))}{2^m} = r^m
\end{equation}
so that the $r$-energy functional of $\mathfrak{T}^m$ writes as
\begin{equation}\label{eq:energy}
    \tilde{E}_{r, \mathfrak{T}^m}
    (f)
=
    \frac{1}{r^m}
    \int_{\mathbb{T}^m}
    \int_{\tilde{Q}_r(x)}
        \left(
            \frac
            {f(x)-f(y)}
            {r}
        \right)^2
        \di\mathbb{L}^m(y)
    \di\mathbb{L}^m(x)
\end{equation}
for any $f \in L^2(\mathbb{T}^m,\mathbb{L}^m)$.

     We act by contradiction. Assume that there exist $r_n \rightarrow 0$ and $\{f_n\} \subset L^2(\mathbb{T}^m,\mathbb{L}^m)$ satisfying $\int_{\mathbb{T}^m} f_n \di \mathbb{L}^m = 0$ and $\|f_n\|_{L^2(\mathbb{T}^m)}=1$, such that
    \begin{equation}
    \label{eq:ContradictHypEigValue1}
        -\tilde{\Delta}_{r_n}f_n
    =
        \lambda_{1}(-\tilde{\Delta}_{r_n})f_n,
    \quad\quad
        \lambda_{1}(-\tilde{\Delta}_{r_n})
    \rightarrow
        0.
    \end{equation}
    With no loss of generality, we assume that each $r_n$ is small enough to ensure that $\tilde{E}_{r_n, \mathfrak{T}^m}$ writes as in \eqref{eq:energy}.
    
    From \eqref{eq:meas} we can write, for any $n$ and $\mathbb{L}^m$-a.e.~$x \in \mathbb{T}^m$,
    \begin{align}
        -\tilde{\Delta}_{r_n}f_n(x)
    &=
    \frac{1}{r_n^2}
    \left(
        f_n(x)
    -
        \frac{1}{r_n^m}
        \int_{\mathbb{T}^m}
            1_{\tilde{Q}_{r_n}(x)}(y)
            f_n(y)
            \di \mathbb{L}^m(y)
    \right) \nonumber \\
    &=
        \frac{1}{r_n^2}
        \left(
        f_n(x)
    -
        \frac{1}{r_n^m} \bigg(
            1_{\tilde{Q}_{r_n}(0)}
            *
            f_n  \bigg)
            (x)
        \right).
        \label{eq:convolutionLaplaceEq}
    \end{align}
    We consider the Fourier decomposition of $f_n$, $1_{\tilde{Q}_{r_n}(0)}$, and $-\tilde{\Delta}_{r_n}f_n$, namely
    \begin{equation*}
        f_n 
    =
        \sum_{p \in \mathbb{Z}^m}
        a_{p,n} e_p 
    \end{equation*}
        \begin{equation*}
        1_{\tilde{Q}_{r_n}(0)}
    =
        \sum_{p \in \mathbb{Z}^m}
        b_{p,n} e_p
    \end{equation*}
    \begin{equation*}
        -\tilde{\Delta}_{r_n}f_n 
    =
        \sum_{p \in \mathbb{Z}^m}
        c_{p,n} e_p
    \end{equation*}
    where $\{e_p\}_{p \in \mathbb{Z}^m}$ is the orthonormal basis of $L^2(\mathbb{T}^m,\mathbb{L}^m)$ given by  \[
    e_p : \mathbb{T}^m \ni x \mapsto e^{i \pi p\cdot \pi^{-1}(x)} \qquad \forall \, p \in \mathbb{Z}^m.
    \]
    Since the Fourier coefficients of a convolution are the product of the coefficients, we obtain from \eqref{eq:convolutionLaplaceEq} that
    \begin{equation}
        c_{p,n}
    =
        \frac{1}{r_n^2}
        \left(
            a_{p,n}
        -
            \frac{b_{p,n}}{r_n^m}
            a_{p,n}
        \right)
    =
        \frac{a_{p,n}}{r_n^2}
        \left(
            1
        -
            \frac{b_{k,n}}{r_n^m}
        \right).
    \end{equation}
    We can compute each coefficient $b_{p,n}$ by means of Fubini's Theorem; we obtain
    \begin{equation}
        b_{p,n}
    =
        \int_{Q_{r_n}(0)}
        e^{i\pi p \cdot x} \frac{d\mathcal{L}^m(x)}{2^m}
    =
        r_n^{m}\prod_{i \in J(p)} \sinc(p_i r_n).
    \end{equation}
    Thus
    \begin{equation}
        c_{p,n}
    =
        \frac{a_{p,n}}{r_n^2}
        \left(
            1
        -
            \prod_{i \in J(p)}
        \sinc(p_i r_n)
        \right).
    \end{equation}
    Using Lemma \ref{SinLemma}, for $p \in \mathbb{Z} \backslash \{0\}$ we conclude that there exists $\alpha > 0$ such that
    \begin{equation}
    \frac{1}{r_n^2}
        \left|
            1
        -
            \prod_{i \in J(p)}
        \sinc(p_i r_n)
        \right|
    \geq
        \alpha.
    \end{equation}
    This implies that
    \begin{equation}
        |c_{p,n}|^2
    \geq
        |a_{p,n}|^2\alpha^2.
    \end{equation}
    By Parseval's identity, and since $f_n \in \Pi(\mathbb{T}^m, \mathbb{L}^m)$ we have $a_{0,n} = 0$,
    \begin{equation}
        \|-\tilde{\Delta}_{r_n}f_n\|_{L^2(\mathbb{T}^m)}^2
    =
        \sum_{p \in \mathbb{Z}^m}
        |c_{p,n}|^2
    \geq
        \sum_{p \in \mathbb{Z}^m\backslash\{0\}}
        |a_{p,n}|^2\alpha^2
    =
        \alpha^2
        \|f_n\|_{L^2(\mathbb{T}^m)}^2,
    \end{equation}
    in contradiction with $\|-\tilde{\Delta}_{r_n}f_n\|_{L^2(\mathbb{T}^m)} \rightarrow 0$ provided by \eqref{eq:ContradictHypEigValue1}.
\end{proof}

\subsection{ Shrinking hypercubes} For any $b>0$, consider the metric measure space $\mathfrak{Q}^m(b)\df([0,b]^m,\dist_\infty,\mathcal{L}^m)$. It trivially  satisfies the assumptions of Theorem \ref{th:1}, so that for any small enough $r$,
\begin{equation}\label{eq:101}
\tilde{\lambda}_{1,r}=\lambda_{1}(- \tilde{\Delta}_{r,\mathfrak{Q}^m(b)})>0.
\end{equation}
Then the following holds.

\begin{lemma}
\label{lem:firstEigenvalueOfCubes}
    \begin{equation}
    \label{firstEigenvalueExplodes}
        \liminf_{b \rightarrow 0}
        \lim_{r \to 0}
            \lambda_{1}(-\tilde{\Delta}_{r,\mathfrak{Q}^m(b)})
        = +\infty.
    \end{equation}
\end{lemma}
\begin{proof}
    We suppose by contradiction that \eqref{firstEigenvalueExplodes} fails. Then there exist $b_n \rightarrow 0$ and $r_n \rightarrow 0$ such that
    \begin{equation*}
        \lim_n
         \lambda_{1}(-\tilde{\Delta}_{r_n,\mathfrak{Q}^m(b)}) < +\infty.
    \end{equation*}
    Since we are first taking the limit in $r$ and then in $b$, we can assume that $\overline{r}_n := \frac{r_n}{b_n} \rightarrow 0$. For any $n$, by a simple scaling argument we have
    \begin{equation*}
         \lambda_{1}(-\tilde{\Delta}_{r_n,\mathfrak{Q}^m(b_n)})
    =
        \frac{1}{b_n^2}
        \lambda_{1}
        (-\tilde{\Delta}_{\overline{r}_n,\mathfrak{Q}^m(1)})
    \end{equation*}
thus
    \begin{equation}
    \label{eq:contradictLambda1To0}
        \lambda_{1}
        (-\tilde{\Delta}_{\overline{r}_n,\mathfrak{Q}^m(1)})
    \rightarrow
        0.
    \end{equation}
    From \eqref{eq:101}, assuming that each $r_n$ is small enough, we know that there exists $f_n \in \Pi([0,1]^m,\mathcal{L}^m)$ such that $\|f_n\|_{L^2([0,1]^m,\mathcal{L}^m)}=1$ and
    \begin{equation*}
        \tilde{E}_{\overline{r}_n,\mathfrak{Q}^m(1)}(f_n)
    =
         \lambda_{1}
        (-\tilde{\Delta}_{\overline{r}_n,\mathfrak{Q}^m(1)}).
    \end{equation*}
    Consider the continuous function
        $$\begin{array}{cccccc}
    T & : & \mathbb{T}^m & \rightarrow & [0,1]^m \\
      &   &  x           & \mapsto     &  (|\bar{x}_1|,...,|\bar{x}_m|)
    \end{array}$$
    where $\bar{x} =\pi^{-1}(x) \in [-1,1)^m$. For any $n$, set
    \[
    \tilde{f}_n = f_n \circ T \in L^2(\mathbb{T}^m,\mathbb{L}^m).
    \]
    From this we have that $f_n \in \Pi(\mathbb{T}^m, \mathbb{L}^m)$
    Let us prove that
    \[
    \|\tilde{f}_n\|_{L^2(\mathbb{T}^m,\mathbb{L}^m)} = 1.
    \]
    Let $C_1,\ldots,C_{2^m}$ denote the $2^m$ sets of the form $I_1 \times \ldots \times I_m$ where each $I_i$ is either $[-1,0]$ or $[0,1]$. For any $j \in \{1,\ldots,2^m\}$, set
    $$\begin{array}{cccccc}
    N_j & : & [0,1]^m & \rightarrow & C_j \\
      &   &  (\xi_1,\ldots,\xi_m)           & \mapsto     &  (\eps_i\xi_1,\ldots,\eps_m\xi_m)
    \end{array}$$
    where $\eps_i$ is $1$ if $I_i = [0,1]$ and $-1$ otherwise. Note that $N_j$ is an isometry which preserves the Lebesgue measure, and that $T \circ \pi \circ N_j$ is equal to the identity. Then
    \begin{align}\label{eq:computation}        \|\tilde{f}_n\|_{L^2(\mathbb{T}^m,\mathbb{L}^m)}^2 = \int_{\mathbb{T}^m} (f_n \circ T)^2 \di \mathbb{L}^m  & = \frac{1}{2^m}\int_{[-1,1]^m} (f_n \circ T\circ \pi)^2 \di \mathcal{L}^m \nonumber \\
        & = \frac{1}{2^m} \sum_{j=1}^{2^m}\int_{C_j} (f_n \circ T\circ \pi)^2 \di \mathcal{L}^m \nonumber \\
        & = \frac{1}{2^m} \sum_{j=1}^{2^m}\int_{[0,1]^m} (f_n \circ T\circ \pi \circ N_j)^2 \di \mathcal{L}^m  \nonumber \\
        & = \|f_n\|_{L^2([0,1]^m,\mathcal{L}^m)}^2 = 1.
     \end{align}
     
     We claim that
    \begin{equation}\label{eq:claim}
        \tilde{E}_{\overline{r}_n,\mathfrak{T}^m}(\tilde{f}_n) \le 2^m \tilde{E}_{\overline{r}_n,\mathfrak{Q}^m(1)}(f_n).
    \end{equation}
    Since $\tilde{f}_n \in \Pi(\mathbb{T}^m, \mathbb{L}^m)$, the latter provides a contradiction with Proposition \ref{prop:torus}, namely
    \[
    0 < \lambda_1(-\tilde{\Delta}_{\bar{r}_n,\mathfrak{T}^m}) \le  2^m\lambda_{1}
        (-\tilde{\Delta}_{\overline{r}_n,\mathfrak{Q}^m(1)}) \to 0.
    \]
Given $x \in \mathbb{T}^m$, define
    \[
    \tilde{G}_n(x) \df 
    \frac{1}{\overline{r}_n^m}
    \int_{\tilde{Q}_{\overline{r}_n}(x)}
                    \left(
                        \frac{\tilde f_n(x)-\tilde f_n(y)}{\overline{r}_n}
                    \right)^2
                     \di \mathbb{L}^m(y).
                     \]
For any $x \in [0,1]^m$, set
    \begin{align*}
    {G}_n(x) &\df \int_{\tilde{Q}_{\overline{r}_n}(x)}
                    \left(
                        \frac{1}{V(x,\overline{r}_n)}
                        +
                        \frac{1}{V(y, \overline{r}_n)}
                   \right)\\
            &\qquad\qquad\quad\times
                    \left(
                        \frac{ f_n(x)- f_n(y)}{\overline{r}_n}
                    \right)^2
                     \di \mathcal{L}^m(y)
    \end{align*}
        where
\begin{equation}
    V(z,r) \df \mathcal{L}^m(Q_r(z)\cap [0,1]^m)
\end{equation}
for any $z\in [0,1]^m$. Then
    \begin{align*}
        4 \tilde{E}_{\overline{r}_n,\mathfrak{T}^m}(\tilde{f}_n)
    &=
        \int_{\mathbb{T}^m}
                \tilde{G}_n(x)
            \di
            \mathbb{L}^m(x)\\
        4 \tilde{E}_{\overline{r}_n,\mathfrak{Q}^m(1)}(f_n)
    &=
        \int_{[0,1]^m}
                {G}_n(x)
            \di
            \mathcal{L}^m(x).
    \end{align*}
     Moreover,  for all $x \in \mathbb{T}^m$,
    \begin{equation}
    \label{eq:ineqBetweenTorVolAndCuveVol}
        \frac{1}{\overline{r}_n^m}
    \leq
        \frac{2^m}{\mathcal{L}^m(Q_{\overline{r}_n}(T(x)) \cap [0,1]^m)}
    =
        \frac{2^m}{V(T(x),\overline{r}_n)}.
    \end{equation}

        For any $x \in \mathbb{T}^m$, there exists some $k \in \{1,...,2^m\}$ and $\overline{x} \in C_k$ such that $\pi(\overline{x}) = x$. We will now consider for each $j \in \{1,...,2^m\}$ the rectangle given by
        \begin{equation}
            R_{n,j}(x)
        :=
        T(
            \pi(C_j)
            \cap
            \tilde{Q}_{\overline{r}_n}(x)
        )
        \subset 
        [0,1]^m.
        \end{equation}
        and we point out that
        \begin{equation}
        \label{eq:inclusionOfRectangles}
            R_{n,j}(x)
        \subset
            R_{n,k}(x)
        =
            [0,1]^m
            \cap 
            Q_{\overline{r}_n}(T(x))
        \end{equation}
        for any $j \in \{1,...,2^m\}$. This follows since for each $i \in \{1,...,m\}$ we have 4 possibilities
        \begin{enumerate}
            \item $|\overline{x}_i|<\overline{r}_n$
            \item $|1-\overline{x}_i|<\overline{r}_n$
            \item $|-1-\overline{x}_i|<\overline{r}_n$
            \item $\neg\left((1)\lor (2)\lor(3) \right)$.
        \end{enumerate}
        If $\pi_i:\mathbb{R}^m \rightarrow \mathbb{R}$ is the projection in the $i$-th coordinate, we conclude
        \begin{equation}
        \pi_i(R_{n,k}(x))
    =
        \begin{cases}
        [0,\overline{r}_n+|\overline{x}_i|]  & \text{ if }(1)\\
        [-\overline{r}_n+|\overline{x}_i|,1]  & \text{ if }(2)\\
        [-\overline{r}_n+|\overline{x}_i|,1] & \text{ if }(3)\\
          [-\overline{r}_n,\overline{r}_n]  & \text{ if }(4)
        \end{cases}
        \end{equation}
        \begin{equation}
        \pi_l(R_{n,k}(x))
    =
        \begin{cases}
        \pi_i(R_{n,l}(x)) = \pi_i(R_{n,k}(x))  & \text{ if } \pi_i(C_j) = \pi_i(C_k)\\
        [0,\overline{r}_n-|\overline{x}_i|]  & \text{ if }(1) \text{ and } \neg\left(\pi_i(C_j) = \pi_i(C_k)\right)\\
        [2-\overline{r}_n-|\overline{x}_i|,1]  & \text{ if }(2)\text{ and } \neg\left(\pi_i(C_j) = \pi_i(C_k)\right)\\
        [2-\overline{r}_n-|\overline{x}_i|,1] & \text{ if }(3)\text{ and } \neg\left(\pi_i(C_j) = \pi_i(C_k)\right)\\
          \emptyset  & \text{ if }(4)\text{ and } \neg\left(\pi_i(C_j) = \pi_i(C_k)\right).
        \end{cases}
        \end{equation}
        Thus we conclude that for all $i \in \{1,...,m\}$ and $l \in \{1,...,2^m\}$ we have $\pi_i(R_{n,l}(x)) \subset \pi_i(R_{n,k}(x))$, and since these sets are rectangles, we conclude equation \eqref{eq:inclusionOfRectangles}.

        From \eqref{eq:ineqBetweenTorVolAndCuveVol} we can deduce
        \begin{align}
            \tilde{G}_n(x)
    &=
        \frac{1}{\bar{r}_n^m}
        \int_{\tilde{Q}_{\overline{r}_n}(x)}
                    \left(
                        \frac{\tilde f_n(x)-\tilde f_n(y)}{\overline{r}_n}
                    \right)^2
                     \di \mathbb{L}^m(y)\\
    &\leq
        \int_{\tilde{Q}_{\overline{r}_n}(x)}
                    \left(
                        \frac{2^m}{V(T(x), \overline{r}_n)}
                        +
                        \frac{2^m}{V(T(y), \overline{r}_n)}
                   \right)\\
            &\quad\quad\times
                    \left(
                        \frac{ f_n(T(x))- f_n(T(y))}{\overline{r}_n}
                    \right)^2
                     \di \mathbb{L}^m(y)
                \\
        &=
            \sum_{j=1}^{2^m}
            \int_{\tilde{Q}_{\overline{r}_n}(x)\cap \pi(C_j)}
                    \left(
                        \frac{2^m}{V(T(x), \overline{r}_n)}
                        +
                        \frac{2^m}{V(T(y), \overline{r}_n)}
                   \right)\\
            &\quad\quad\times
                    \left(
                        \frac{ f_n(T(x))- f_n(T(y))}{\overline{r}_n}
                    \right)^2
                     \di \mathbb{L}^m(y)\\
        &=
            \frac{1}{2^m}
            \sum_{j=1}^{2^m}
            \int_{\pi^{-1}(\tilde{Q}_{\overline{r}_n}(x))\cap C_j}
                    \left(
                        \frac{2^m}{V(T(x),\overline{r}_n)}
                        +
                        \frac{2^m}{V(T(\pi(y)), \overline{r}_n)}
                   \right)\\
            &\quad\quad\times
                    \left(
                        \frac{ f_n(T(x))- f_n(T(\pi(y)))}{\overline{r}_n}
                    \right)^2
                     \di \mathcal{L}^m(y)
        \end{align}
        For each integral, change coordinates by $N_j$ to conclude
        \begin{align}
            \tilde{G}_n(x)
        &\leq
            \frac{1}{2^m}
            \sum_{j=1}^{2^m}
            \int_{N_j^{-1}(\pi^{-1}(\tilde{Q}_{\overline{r}_n}(x))\cap C_j)}
                    \left(
                        \frac{2^m}{V(T(x),\overline{r}_n)}
                        +
                        \frac{2^m}{V(T(\pi(N_j(y))), \overline{r}_n)}
                   \right)\times\\
            &\quad\quad\times
                    \left(
                        \frac{ f_n(T(x))- f_n(T(\pi(N_j(y))))}{\overline{r}_n}
                    \right)^2
                     \di \mathcal{L}^m(y).
        \end{align}
        We have that
        \begin{equation}
            N_j^{-1}(\pi^{-1}(\tilde{Q}_{\overline{r}_n}(x))\cap C_j)
        =
            R_{n,j}(x),
        \end{equation}
        so by equation \eqref{eq:inclusionOfRectangles} and the fact that $T\circ\pi\circ N_j = id$, we conclude
        \begin{align}
            \tilde{G}_n(x)
        &\leq
            \frac{1}{2^m}
            \sum_{j=1}^{2^m}
            \int_{R_{n,j}(x)}
                    \left(
                        \frac{2^m}{V(T(x),\overline{r}_n)}
                        +
                        \frac{2^m}{V(T(y),\overline{r}_n)}
                   \right)\times\\
            &\quad\quad\quad\quad\times
                    \left(
                        \frac{ f_n(T(x))- f_n(y)}{\overline{r}_n}
                    \right)^2
                     \di \mathcal{L}^m(y)\\
            & \leq \frac{1}{2^m}
            \sum_{j=1}^{2^m}
            \int_{R_{n,k}(x)}
                    \left(
                        \frac{2^m}{V(T(x),\overline{r}_n)}
                        +
                        \frac{2^m}{V(y,\overline{r}_n)}
                   \right)\times\\
            &\quad\quad\quad\quad\times
                    \left(
                        \frac{ f_n(T(x))- f_n(y)}{\overline{r}_n}
                    \right)^2
                     \di \mathcal{L}^m(y)\\
            & =   \int_{[0,1]^m \cap Q_{\overline{r}_n}(T(x))}
                    \left(
                        \frac{2^m}{V(T(x),\overline{r}_n)}
                        +
                        \frac{2^m}{V(y,\overline{r}_n)}
                   \right)\times\\
            &\quad\quad\quad\quad\times
                    \left(
                        \frac{ f_n(T(x))- f_n(y)}{\overline{r}_n}
                    \right)^2
                     \di \mathcal{L}^m(y)\\
            &=
                2^m
                G_n(T(x)).
        \end{align}
    Now we integrate both sides in $\mathbb{T}^m$ and change variables by $\pi$ and $N_j$ to conclude
    \begin{align}
        4\tilde{E}_{\overline{r}_n, \mathfrak{T}^m}(\tilde{f}_n)
    =
        \int_{\mathbb{T}^m}
            \tilde{G}_n(x)
            \di\mathbb{L}^m(x)
    &\leq
        2^m
        \int_{\mathbb{T}^m}
            G_n(T(x))
        \di\mathbb{L}^m(x)\\
    &=
        \int_{[-1,1)^m}
            G_n(T(\pi(x)))
        \di\mathcal{L}^m(x)\\
    &=
        \sum_{j=1}^{2^m}
        \int_{C_j}
            G_n(T(\pi(x)))\di\mathcal{L}^m(x)\\
    &=
        \sum_{j=1}^{2^m}
        \int_{[0,1]^m}
            G_n(T(\pi(N_j(x))))\di\mathcal{L}^m(x)\\
    &=
         \sum_{j=1}^{2^m}
        \int_{[0,1]^m}
            G_n(x)\di\mathcal{L}^m(x)\\
    &=
        2^m
        4\tilde{E}_{\overline{r}_n,[0,1]^m}(f_n) = 2^m4\lambda_1(-\tilde{\Delta}_{\overline{r}_n,\mathfrak{Q}^m(1)}).
    \end{align}
    With this we obtain \eqref{eq:claim}.
\end{proof}

\section{$L^2$ convergence}\label{sec:manifolds}

In this section, we prove Theorem \ref{th:2}.  Let $M$ be a smooth, compact, connected manifold of dimension $m \ge 2$. Assume that $M$ is endowed with a smooth Riemannian metric $g$ and let $\dist_g$ and $\vol_g$ be the associated Riemannian distance and volume measure on $M$.

In this smooth context, the function $V(\cdot,r)$ is obviously continuous for any $r>0$. Since $M$ is compact,  this implies that the metric measure space $(M,\dist_g,\vol_g)$ satisfies \eqref{eq:V-1_int} and \eqref{I}. Then Lemma \ref{lem:sAMV} applies and ensures that $\tilde{\Delta}_r$ is a bounded self-adjoint operator acting on $L^2(M,\vol_g)$. The compactness of $M$ also ensures that $(M,\dist_g,\vol_g)$ is locally Ahlfors regular: there exists a constant $C>1$ such that for any $x \in M$ and $r \in (0,\mathrm{diam}(M)]$,
    \begin{equation}\label{eq:ahlforsM}
     C^{-1}r^m \le V(x,r) \le Cr^m.
    \end{equation}
Note that this condition trivially implies a uniform local doubling property for $(M,\dist_g,\vol_g)$.

\subsection{Convergence in the sense of distributions}
Recall that $C_m$ is defined in \eqref{eq:C_m}.  For any $x \in M \backslash \partial M$, we let $\exp_x$ be the exponential map centered at $x$. We identify $T_x M$ with $\mathbb{R}^m$ and write $\mathbb{B}^m_{r}(v)$ for the Euclidean ball in $\mathbb{R}^m$ centered at $v$ with radius $r>0$.  Then there exists $\delta>0$ such that the restriction of $\exp_x$ to $\mathbb{B}^m_{\delta}(0)$ is a diffeomorphism onto its image; recall that the injectivity radius $\mathrm{i}_M(x)$ of $M$ at $x$ is the supremum of the set of such numbers $\delta$. We let $J_x$ be the Radon-Nikodym derivative of the measure $(\exp_x^{-1})_{\#} \vol_g$ with respect to the Lebesgue measure $\mathcal{L}^m$.  It is well-known that for any $\xi \in \mathbb{B}_{\mathrm{i}_M}^m(0)$,
\[
J_x(\xi) = 1 + \underline{O}_K(|\xi|^2)
\]
where for any $h>0$, the notation $\underline{O}_K(h)$ stands for a quantity independent on $x \in K$ whose absolute value divided by $h$ is bounded. Here $K$ is a compact subset of $M$. We write $\underline{O}$ instead of $\underline{O}_M$. Then the following holds.

\begin{prop}
\label{convergenceOfInnerProds}
    Consider $f,\psi \in C^2(M)$. Then
\begin{equation}\label{eq:conv_inn_prod}
        \lim_{r\rightarrow 0}
        \langle 
            -\tilde{\Delta}_r f,\psi 
        \rangle_2
    =
        C_m
        \int_M
        \langle
            \di f
            ,
            \di \psi
        \rangle_g \di \vol_g.
    \end{equation}
\end{prop}

\begin{proof}
    For any $x \in M$ and $r>0$, set
    \[
    \tilde{e}_r(f,\psi;x) \df \frac{1}{4}\fint_{B_r(x)} \left( 1+ \frac{V(x,r)}{V(y,r)}\right)  \frac{(f(x)-f(y))(\psi(x)-\psi(y))}{r^2}\di \vol_g(y)
    \]
    so that
    \[
    \langle 
            -\tilde{\Delta}_r f,\psi 
        \rangle_2 = \int_M \tilde{e}_r(f,\psi;y)\di \vol_g(y).
    \]
    On one hand,
    \begin{align*}
        |\tilde{e}_r(f,\psi;x)| & \le \frac{1}{4}\fint_{B_r(x)} \left( 1+ \frac{V(x,r)}{V(y,r)}\right)  \frac{|f(x)-f(y)||\psi(x)-\psi(y)|}{r^2}\di \vol_g(y)\\
        & \le \frac{1}{4}\fint_{B_r(x)} \left( 1+ \frac{V(x,r)}{V(y,r)}\right)  \frac{\Lip(f)\Lip(\psi)\dist_g^2(x,y)}{r^2}\di \vol_g(y) \\
        & \le \frac{\Lip(f)\Lip(\psi)}{4}\fint_{B_r(x)} \left( 1+ \frac{V(x,r)}{V(y,r)}\right)  \di \vol_g(y).
    \end{align*}
    By \eqref{eq:ahlforsM}, we obtain
    \begin{equation}\label{eq:Lipschitz_bound}
        |\tilde{e}_r(f,\psi;x)| \le \frac{\Lip(f)\Lip(\psi)(1+C^2)}{4}\, \cdot 
    \end{equation}
    On the other hand,  assume that $r$ is smaller than $\mathrm{i}_M(x)$, and consider $\tilde{f} \df f \circ \exp_x$ and  $\tilde{\psi} \df \psi \circ \exp_x$ on $\mathbb{B}^m_r(0)$. The first-order Taylor expansion of $\tilde{f}$ and $\tilde{\psi}$ yields 
    \begin{align*}
   \tilde{f}(\xi) & =  \tilde{f}(0) + (\di \tilde{f})_0(\xi) + \underline{O}_{\{0\}}(|\xi|^2)\\
   \tilde{\psi}(\xi) & =  \tilde{\psi}(0) + (\di \tilde{\psi})_0(\xi) + \underline{O}_{\{0\}}(|\xi|^2).
    \end{align*}
    Then
    \begin{align*}
    & \phantom{=}  \int_{B_r(x)} (f(x)-f(y))(\psi(x)-\psi(y))\di \vol_g(y)\\
    & =  \int_{\mathbb{B}^m_r(0)} (\tilde{f}(0)-\tilde{f}(\xi))(\tilde{\psi}(0)-\tilde{\psi}(\xi))J(\xi)\di \mathcal{L}^m(\xi)\\
    & =  \int_{\mathbb{B}^m_r(0)} ((\di \tilde{f})_0(\xi) + \underline{O}_{\{0\}}(r^2))((\di \tilde{\psi})_0(\xi) + \underline{O}_{\{0\}}(r^2))(1+\underline{O}_{\{0\}}\di \mathcal{L}^m(\xi)\\
    &=
        \sum_{i,j=1}^m
        [(\di \tilde{f})_0]_i
        [(\di \tilde{\psi})_0]_j
        \int_{\mathbb{B}^m_r(0)}
        \xi_j
        \xi_i
        \di \mathcal{L}^m(\xi)
    +
    \underline{O}_{\{0\}}(r^{m+3})\\
     & = (\di \tilde{f})_0 \cdot (\di \tilde{\psi})_0 \int_{\mathbb{B}^m_r(0)} \xi_1^2\di \mathcal{L}^m(\xi) + \underline{O}_{\{0\}}(r^{m+3}).
    \end{align*}
    Moreover,   it is known that (see e.g.~\cite[Remark 2.11]{MT2})
\[
1+ \frac{V(x,r)}{V(y,r)} = {2} + \underline{O}_{\{x\}}(r^2)
\]    
and since $V(x,r)/\mathcal{L}^m(\mathbb{B}^m_r(0))) \to 1$ as $r \downarrow 0$, we obtain
    \begin{align*}
    & \phantom{=}  \fint_{B_r(x)} \left( 1+ \frac{V(x,r)}{V(y,r)}\right) (f(x)-f(y))(\psi(x)-\psi(y))\di \vol_g(y)\\
    & =  \frac{\mathcal{L}^m(\mathbb{B}^m_r(0))({2} + \underline{O}_{\{x\}}(r^2)) }{V(x,r)}\fint_{\mathbb{B}^m_r(0)}  (\tilde{f}(0)-\tilde{f}(\xi))(\tilde{\psi}(0)-\tilde{\psi}(\xi))J(\xi)\di \mathcal{L}^m(\xi)\\
        & =  (2 + \underline{O}_{\{x\}}(r^2)) \left(   (\di \tilde{f})_0 \cdot (\di \tilde{\psi})_0 \fint_{\mathbb{B}^m_r(0)} \xi_1^2\di \mathcal{L}^m(\xi) + \underline{O}_{\{0\}}(r^{3})\right).
    \end{align*}
Since $(\di \tilde{f})_0 \cdot (\di \tilde{\psi})_0 = \langle
            \di f
            ,
            \di \psi
        \rangle_g (x)$ and 
        \[
        \fint_{\mathbb{B}^m_r(0)} \xi_1^2\di \mathcal{L}^m(\xi)  = 2 r^2 C_m
        \]
        by change of variable $\xi \leftrightarrow \eta/r^2$, we eventually obtain that
    \begin{equation}
        \tilde{e}_r(f,\psi;x) = C_m \langle
            \di f
            ,
            \di \psi
        \rangle_g (x) + \underline{O}_{\{x\}}(r) \qquad \text{as $r \to 0$.}
     \end{equation}
By \eqref{eq:Lipschitz_bound} and the compactness of $M$, we can apply the  dominated convergence theorem to the functions  $\tilde{e}_r(f,\psi;\cdot)$. Then we get \eqref{eq:conv_inn_prod}.
\end{proof}

Using integration by parts in \eqref{eq:conv_inn_prod}, we immediately obtain the following.

\begin{corollary}
Let $\partial g$ be the Riemannian metric induced by $g$ on $\partial M$. For any $f \in \cC^\infty(M)$, the following convergence holds in the sense of distributions as $r \downarrow 0$:
 \[    
    (\tilde{\Delta}_r f) \, \vol_g
    \to  C_m \bigg( (\Delta_g f) \, \vol_g + (\partial_\nu^g f) \, \vol_{\partial g}\bigg).
    \]
    
\end{corollary}

\subsection{Pointwise convergence}
We aim to prove Theorem \ref{th:2} in a similar way as Proposition \ref{convergenceOfInnerProds}, that is to say, by means of the dominated convergence theorem. To this aim, we first establish that pointwise convergence holds $\vol_g$-a.e.~on $M$. We recall that $\partial M$ is a $\vol_g$-negligible subset of $M$.

\begin{prop}
\label{lem:LaplaceBeltramiLemma}
    Let $f \in C^\infty(M)$. Then for any $x \in M - \partial M$,
    \begin{equation}
        \lim_{r \rightarrow 0}
        \tilde{\Delta}_r f(x)
    =
         C_m\Delta_g f(x).
    \end{equation}
    Moreover, the convergence is uniform on any compact subset of $M - \partial M$.
\end{prop}
\begin{proof}
Let $K$ be a compact subset of $M - \partial M$. Consider $x \in K$ and $r \in (0,\mathrm{i}_M(x))$.  Set $\tilde{f}_x \df f \circ \exp_x$. Acting like in the proof of Proposition \ref{convergenceOfInnerProds}, we get
\begin{equation}
\tilde{\Delta}_r f(x)
    = 
     \frac{(2 + \underline{O}_{K}(r^2)) }{2 r^2}
    \fint_{\mathbb{B}^m_r(0)}
        \left(\tilde{f}_x(\xi)
            -
            \tilde{f}_x(0)
            \right)
             \di\mathcal{L}^m(\xi).
\end{equation}
The second-order Taylor expansion of $\tilde{f}_x$ yields 
    \begin{align*}
   \tilde{f}_x(\xi) & =  \tilde{f}_x(0) + (\di \tilde{f}_x)_0(\xi) + \frac{1}{2}  (\di^{(2)} \tilde{f}_x)_0(\xi,\xi) + \underline{O}_K(|\xi|^3)
    \end{align*}
hence we get
\begin{align*}
\fint_{\mathbb{B}^m_r(0)}
      \tilde{f}_x(\xi)
            -
            \tilde{f}_x(0)
            \di\mathcal{L}^m(\xi)
    & = 
    \fint_{\mathbb{B}^m_r(0)}
        (\di \tilde{f}_x)_0(\xi)  \di\mathcal{L}^m(\xi) \\
        & + \frac{1}{2}  \fint_{\mathbb{B}^m_r(0)} (\di^{(2)} \tilde{f}_x)_0(\xi,\xi)  \di\mathcal{L}^m(\xi) + \underline{O}_K(r^3).
\end{align*}
The first term vanishes by symmetry. The second term is equal to
\[
\frac{1}{2} \Delta \tilde{f}_x(0)  \fint_{\mathbb{B}^m_r(0)} \xi_1^2  \di\mathcal{L}^m(\xi) = \Delta_g f(x) r^2 C_m.
\]
In the end we get
\begin{align*}
\tilde{\Delta}_r f(x)
    & = 
     \frac{2 + \underline{O}_K(r^2)}{2 r^2}
    \left(  \Delta_g f(x) r^2 C_m + \underline{O}_K(r^3)\right) \\
    & =  (1 + \underline{O}_K(r^2))
    \left( C_m \Delta_g f(x)  + \underline{O}_K(r)\right)
\end{align*}
from which follows the desired result, by letting $r \downarrow 0$.
\end{proof}

\subsection{Uniform bound}

We wish now to provide a uniform $L^\infty$ bound for the functions $-\tilde{\Delta}_r \psi$ where $r$ is in a neighborhood of zero.

Let us first consider the case $\partial M = \emptyset$. From Proposition \ref{lem:LaplaceBeltramiLemma}, we have uniform convergence
\[
\|\tilde{\Delta}_r f - C_m \Delta_g f\|_\infty \to 0
\]
so that
\[
\|\tilde{\Delta}_r f - C_m \Delta_g f\|_2 \le
\|\tilde{\Delta}_r f - C_m \Delta_g f\|_\infty \vol_g(M) \to 0.
\]
Let us now deal with the case $\partial M \neq \emptyset$.

\begin{prop}
\label{prop:L_infty_bound}
    Assume that $\partial M \neq \emptyset$. Consider $f \in C^\infty_{\nu}(M)$. Then there exists $r_0>0$ such that
    \begin{equation}\label{eq:L_infty_bound}
    \sup_{0<r<r_0}\|\tilde{\Delta}_r f\|_{L^\infty(M)} < +\infty.
    \end{equation}
\end{prop}

\begin{proof}
Since $\partial M$ is compact, we can find a finite collection of smooth parameterizations $\psi_i:(-4,4)^{m-1} \rightarrow \partial M$ such that
\begin{equation}
\label{eq:coverOfBoundary}
    \partial M = \bigcup_i
    \psi_i \left(\left[ -1,1\right]^{m-1}\right).
\end{equation}

\paragraph{\bf Step 1.} We work with any of the previous $\psi_i$ which we denote by $\psi$. For any $x \in \partial M$, let $\nu_x \in T_xM$ be the unit inner normal vector of $\partial M$ at $x$. Since $\partial M$ is smooth, there exists $\epsilon > 0$ such that the map $E: \partial M \times [0,\epsilon] \rightarrow M$ given by
\begin{equation}
    E(x, t)
=
    \exp_x^M(t \nu_x)
\end{equation}
is an embedding, and there exists a smooth family of metrics $\{g_t\}_{t \in [0,\eps]}$ on $\partial M$ such that for any $(x,t) \in \partial M \times [0,\epsilon]$,
\begin{equation}\label{eq:E*g}
    (E^*g)_{(x,t)}
=
    (g_t
\oplus
    d\tau^2)_{(x,t)}.
\end{equation}
Pulling back each metric $g_t$ by $\psi$ we have, for any $\xi \in (-4,4)^{m-1}$ and $v,w \in \mathbb{R}^{m-1}$,
\begin{equation}
\label{eq:pullBackOfBoundaryMetrics}
    (\psi^*g_t)_\xi(v,w)
=
    v^T \cdot A_{(\xi,t)} \cdot w,
\end{equation}
for some positive definite, symmetric $(m-1)$-square matrix $A_{(\xi,t)}$. From the non-degeneracy of the metric and a Lipschitz bound, we have that there exist $C,\tilde{c}>0$ such that for all $t,s \in [0,\epsilon], \xi \in [-3,3]^{m-1}, v \in \mathbb{R}^m$ we have
\begin{equation}
\label{eq:nonDegeneracyMetric}
    \left[(\psi^*g_t) \oplus d\tau^2\right]_{(\xi,t)}(v,v)
\geq
    \tilde{c}|v|^2,
\end{equation}
\begin{equation}
\label{eq:lipschitzOfMetric}
    \left|
        \left[(\psi^*g_t) \oplus d\tau^2\right]_{(\xi,t)}(v,v)
    -
        \left[(\psi^*g_s) \oplus d\tau^2\right]_{(\xi,t)}(v,v)
    \right|
\leq
    C|v|^2 |t-s|.
\end{equation} 

\begin{claim}
\label{lem:estimateOfDifferenceMetrics}
    There exists $K>0$ such that for any $t,s \in [0,\epsilon]$, $\xi \in [-3,3]^{m-1},$ $v \in \mathbb{R}^m$ such that
    \begin{equation}
    \label{eq:differenceBetweenMetrics}
        \left|
            O(\xi,t,s,v)
        \right|
    \leq
        K|v|\cdot |t-s|,
    \end{equation}
    where \begin{equation}
\label{eq:DefinitionOfO}
    O(\xi,t,s,v)
=
    \left[(\psi^*g_t) \oplus d\tau^2\right]_{(\xi,t)}^\frac{1}{2}(v,v)
    -
        \left[(\psi^*g_s) \oplus d\tau^2\right]_{(\xi,t)}^\frac{1}{2}(v,v).
\end{equation}
\end{claim}
\begin{proof}
    Consider $\tilde{c}>0$ given by \eqref{eq:nonDegeneracyMetric}. We know that there exists $M>0$ such that the map $\sqrt{\cdot}: [\tilde{c},+\infty) \rightarrow \mathbb{R}$ is $M$-Lipschitz.
    By homogeneity in $|v|$ of \eqref{eq:differenceBetweenMetrics}, we can assume that $|v| = 1$. Then the Lipschitz condition and \eqref{eq:lipschitzOfMetric} yield
    \begin{align}
        \left|O(\xi,t,s,v)\right|
    & =
    \left|
        \left[(\psi^*g_t) \oplus d\tau^2\right]_{(\xi,t)}^\frac{1}{2}(v,v)
        -
            \left[(\psi^*g_s) \oplus d\tau^2\right]_{(\xi,t)}^\frac{1}{2}(v,v)
    \right|\\
    & \leq
    M\left|
        \left[(\psi^*g_t) \oplus d\tau^2\right]_{(\xi,t)}(v,v)
        -
            \left[(\psi^*g_s) \oplus d\tau^2\right]_{(\xi,t)}(v,v)
    \right|\\
    & \leq 
    MC|v|^2|t-s|
    =
    MC|t-s|.
    \end{align}
\end{proof}

\paragraph{\bf Step 2.} Let $T(\epsilon) \df E(\partial M \times [0,\epsilon])$ be the $\epsilon$ tubular neighborhood of the boundary. Then $E$ is a diffeomorphism between $\partial M \times [0,\epsilon]$ and $T(\epsilon)$. For fixed $s\in [0,\epsilon]$, consider the product metric $g_s \oplus d\tau^2$ in $\partial M \times [0,\epsilon]$ and define the metric in $T(\epsilon)$ 
\begin{equation}
    \eta_s \df (E^{-1})^*(g_s\oplus d\tau^2).
\end{equation}
Let $\dist_s$ be the distance induced from this metric. Consider the map $\Phi:(-4,4)^{m-1}\times [0,\epsilon] \rightarrow M$ given by
\begin{equation}\label{eq:defPhi}
    \Phi(\xi,t)
=
    E(\psi(\xi),t)
\end{equation}
and note that \eqref{eq:E*g} implies that for any $(\xi,t) \in (-4,4)^{m-1}\times [0,\epsilon]$,
\begin{equation}\label{eq:Phi}
    (\Phi^*g)_{(\xi,t)} = (\psi^*g_t \oplus \di \tau^2)_{(\xi,t)}.
\end{equation}

Set $\mathbb{H}^m \df \{v \in \mathbb{R}^m : v_m \ge 0\}$.  Let $L(\tilde{\gamma})$ be the Euclidean length of a curve $\tilde{\gamma}:[0,1] \to \mathbb{R}^m$.

\begin{claim}
\label{lem:geodesiceInParametrization}
    Let $\tilde{c}>0$ be given by \eqref{eq:nonDegeneracyMetric}. For any $r>0$ and $s \in [0,\eps]$, if a couple $(\xi,t) \in [-2,2]^{m-1}\times[0,\epsilon]$ is such that $\mathbb{B}^m_{r/\sqrt{\tilde{c}}}(\xi,t) \cap \mathbb{H}^m \subset [-3,3]^{m-1}\times[0,\epsilon]$, then the following holds for any $y \in M$.
    \begin{enumerate}
        \item If $\dist(\Phi(\xi,t), y) < r$, then the image of any $\dist$-minimizing geodesic $\gamma : [0,1] \rightarrow M$ is contained in $\Phi(\mathbb{B}^m_{r/\sqrt{\tilde{c}}}(\xi,t) \cap \mathbb{H}^m)$ and $\tilde{\gamma} = \Phi^{-1}\circ \gamma$ satisfies $L(\tilde{\gamma}) < \frac{r}{\sqrt{\tilde{c}}}$.
        \item If $\dist_s(\Phi(\xi,t), y) < r$, then the image of any $\dist_s$-minimizing geodesic $\gamma : [0,1] \rightarrow M$ is contained in $\Phi(\mathbb{B}^m_{r/\sqrt{\tilde{c}}}(\xi,t) \cap \mathbb{H}^m)$ and $\tilde{\gamma} = \Phi^{-1}\circ \gamma$ satisfies $L(\tilde{\gamma}) < \frac{r}{\sqrt{\tilde{c}}}$.
    \end{enumerate}
\end{claim}
\begin{proof}
    We prove the first result only since the proof of the second one follows from similar lines. Consider a $\dist$-minimizing geodesic $\gamma :[0,1] \rightarrow M$ from $\Phi(\xi,t)$ to $y$. Set \begin{align*}
        \delta & \df \sup \left\{t \in [0,1] \, : \, \gamma(s) \in \Phi(\mathbb{B}^m_{r/\sqrt{c}}(\xi,t) \cap \mathbb{H}^m) \text{ for any }s \in [0,t)\right\},\\
        \tilde{\gamma} & \df \Phi^{-1} \circ \left. \gamma \right|_{[0,\delta]},
        \end{align*}
    and observe that $\gamma([0,1]) \subset \Phi(\mathbb{B}^m_{r/\sqrt{c}}(\xi,t) \cap \mathbb{H}^m)$ if and only if $\delta = 1$. We claim that
    \begin{equation}\label{eq:111}
        L(\tilde{\gamma}) \le \frac{\dist_g(\Phi(\xi,t), y) }{\sqrt{\tilde{c}}} \, \cdot 
    \end{equation}
    Indeed, setting  $ (\tilde{\alpha},\tilde{\gamma}_m) \df  \tilde{\gamma}$ where $\tilde{\alpha} :  [0,\delta] \rightarrow [-3,3]^{m-1}$ and $\tilde{\gamma}_m:[0,\delta]\rightarrow[0,\epsilon]$, we have
    \begin{align}
        \dist(\Phi(\xi,t), y)
    & =
        \int_{0}^1g_{\gamma(w)}^\frac{1}{2}(\dot{\gamma}(w),\dot{\gamma}(w)) \di w\\
    & \geq
        \int_0^\delta
        g_{\gamma(w)}^\frac{1}{2}(\dot{\gamma}(w),\dot{\gamma}(w)) \di w\\
    & =
        \int_0^\delta
        (\Phi^*
        g)_{\tilde{\gamma}(w)}^\frac{1}{2}
        (\dot{\tilde{\gamma}}(w), \dot{\tilde{\gamma}}(w)) \di w \\
    & =  \int_0^\delta
        (\psi^*g_{\tilde{\gamma}_m(w)}\oplus d\tau^2)_{\tilde{\gamma}(w)}^\frac{1}{2}
        (\dot{\tilde{\gamma}}(w), \dot{\tilde{\gamma}}(w))\di w \qquad \text{by \eqref{eq:Phi}}\\
    & \geq
        \int_0^\delta
        \sqrt{\tilde{c}}
        \left|\dot{\tilde{\gamma}}(w))\right|
        \di w \qquad \qquad \qquad \qquad \qquad \qquad   \text{by \eqref{eq:nonDegeneracyMetric}} \\
    & =
        \sqrt{\tilde{c}}
        \,
        L(\tilde{\gamma}).
    \end{align}
    Now  we claim that:
    \begin{equation}\label{eq:112}
        \delta < 1 \qquad \Rightarrow \qquad \dist_g(\Phi(\xi,t), y)\geq r.
    \end{equation}
    Indeed, if $\delta < 1$, since $\tilde{\gamma}(0) = (\xi,t)$ and $\tilde{\gamma}(\delta) \in \partial\mathbb{B}^m_{r/\sqrt{\tilde{c}}}(\xi,t)\cap \mathbb{H}^m$, then
    \begin{equation}
        L(\tilde{\gamma}) \geq \frac{r}{\sqrt{\tilde{c}}}
    \end{equation}
    and we get $\dist_g(\Phi(\xi,t), y)\geq r$ from \eqref{eq:111}. Therefore, if $\dist_g(\Phi(\xi,t), y)<r$, then \eqref{eq:112} implies that $\delta=1$ which means $\gamma([0,1])$ is included in $\Phi(\mathbb{B}^m_{r/\sqrt{c}}(\xi,t) \cap \mathbb{H}^m)$, and \eqref{eq:111} yields $L(\tilde{\gamma})< \frac{r}{\sqrt{\tilde{c}}}$ as desired.  
\end{proof}

\begin{claim}
\label{lem:relationBetweenMetrics}
    There exists $K>0$ such that for all $(\xi,t),(\eta,s) \in [-2,2]^{m-1}\times [0,\epsilon]$ and $r>0$ such that $\mathbb{B}^m_{r/\sqrt{\tilde{c}}}(\xi,t) \cap \mathbb{H}^m \subset [-3,3]^{m-1}\times[0,\epsilon]$:
    \begin{equation}
    \label{eq:relationBetweenMetrics1}
        \dist(\Phi(\xi,t), \Phi(\eta,s) )
    <
        r \quad 
    \Rightarrow \quad 
        \dist_t(\Phi(\xi,t), \Phi(\eta,s) )
    <
        r+Kr^2,
    \end{equation}
    \begin{equation}
    \label{eq:relationBetweenMetrics2}
        \dist(\Phi(\xi,t), \Phi(\eta,s) )
    \geq
        r \quad 
    \Rightarrow \quad 
        \dist_t(\Phi(\xi,t), \Phi(\eta,s) )
    \geq
        r
    -
        Kr^2.
    \end{equation}
\end{claim}
\begin{proof}
    Suppose that $\dist(\Phi(\xi,t),\Phi(\eta,s))<r$. Let $\gamma:[0,1] \rightarrow M$ be a $\dist$-minimizing geodesic between $\Phi(\xi,t)$ and $\Phi(\eta,s) $. Then by Claim \ref{lem:geodesiceInParametrization}, we have that $\gamma([0,1]) \subset \Phi(\mathbb{B}^m_{r/\sqrt{\tilde{c}}}(\xi,t) \cap \mathbb{H}^m)$, and by defining $(\tilde{\alpha},\tilde{\gamma}_m) = \tilde{\gamma} \df \Phi^{-1}\circ \gamma$, we have that $L(\tilde{\gamma}) < \frac{r}{\sqrt{\tilde{c}}}$. Thus we obtain
    \begin{align}
        \dist(\Phi(\xi,t),\Phi(\eta,s) )
    &=
        \int_0^1
        g_{\gamma(w)}^\frac{1}{2}(\dot{\gamma}(w),\dot{\gamma}(w))\di w\\
    & = \int_0^1
            g_{\gamma(w)}^\frac{1}{2}(\dot{\gamma}(w),\dot{\gamma}(w)) \di w -
        \int_0^1
        \left(
            \eta_t
        \right)_{\gamma(w)}^\frac{1}{2}(\dot{\gamma}(w),\dot{\gamma}(w)) \di w\\
    & \quad  +
        \int_0^1
        \left(
            \eta_t
        \right)_{\gamma(w)}^\frac{1}{2}(\dot{\gamma}(w),\dot{\gamma}(w))\di w\\    
    &=
        \int_0^1
        \left(
            \psi^*g_{\tilde{\gamma}_m(w)}
        \oplus
            d\tau^2
        \right)_{(\tilde{\alpha}(w), \tilde{\gamma}_m(w))}^\frac{1}{2}(\dot{\tilde\gamma}(w),\dot{\tilde\gamma}(w))\di w\\
    &\quad -
        \int_0^1
        \left(
            \psi^*
            g_t
        \oplus
            d\tau^2
        \right)_{(\tilde{\alpha}(w), \tilde{\gamma}_m(w))}^\frac{1}{2}(\dot{\tilde\gamma}(w),\dot{\tilde\gamma}(w))\di w\\
    & \quad +
        \int_0^1
        \left(
            \eta_t
        \right)_{\gamma(w)}^\frac{1}{2}(\dot{\gamma}(w),\dot{\gamma}(w))\di w
    \\
    &=\int_0^1
        O(\tilde{\alpha}(w),\tilde{\gamma}_m(w), t, \dot{\gamma}(w))
    \di w
        \\
    &\quad+
    \int_0^1
        \left(
            \eta_t
        \right)_{\gamma(w)}^\frac{1}{2}(\dot{\gamma}(w),\dot{\gamma}(w))\di w, \label{eq:equalityOfDist}
    \end{align}
    where we use \eqref{eq:DefinitionOfO} to get the last equality. By Claim \ref{lem:estimateOfDifferenceMetrics}, we have that
    \begin{align}
        \left|
            \int_0^1
        O(\tilde{\alpha}(w), \tilde{\gamma}_m(w), t, \dot{\gamma}(w))
            \di w
        \right|
    &\leq
        \int_0^1
        \left|
            O(\tilde{\alpha}(w), \tilde{\gamma}_m(w), t, \dot{\gamma}(w))
        \right|
        \di w\\
    &\leq
        \int_0^1
            K
            \left|
                \dot{\tilde{\gamma}}(w)
            \right|
            \left|
                t-\tilde{\gamma}_m(w)
            \right|
            \di w.
    \end{align} 
    By Claim \ref{lem:geodesiceInParametrization}, we have that $|t-\tilde{\gamma}_m(s)|< \frac{r}{\sqrt{\tilde{c}}}$ and $L(\tilde{\gamma})< \frac{r}{\sqrt{\tilde{c}}}$, hence we get
    \begin{align}
    \label{eq:estimateOfO}
        \left|
            \int_0^1
        O(\tilde{\alpha}(w), \tilde{\gamma}_m(w), t, \dot{\gamma}(w))
            \di w
        \right|
    &\leq
        K
        \frac{r}{\sqrt{\tilde{c}}}
        \int_0^1
        |\dot{\tilde{\gamma}}(w)|
        \di w  < K\frac{r^2}{\tilde{c}} \, \cdot
    \end{align}
    Thus 
    \begin{align}
        \dist_t(\Phi(\xi,t), \Phi(\eta,s) )
    &\leq
        \int_0^1
        \left(
            \eta_t
        \right)_{\gamma(w)}^\frac{1}{2}(\dot{\gamma}(w),\dot{\gamma}(w))\di w\\
    &\leq
        \dist(\Phi(\xi,t), \Phi(\eta,s))
    +
        \left|
            \int_0^1
            O(\tilde{\alpha}(w), \tilde{\gamma}_m(w),t,\dot{\gamma}(w))
            \di w
        \right|\\
    &\leq
        \dist(\Phi(\xi,t), \Phi(\eta,s))
    +
        K\frac{r^2}{\tilde{c}}\\
    & < r + K\frac{r^2}{\tilde{c}} \, \cdot 
    \end{align}
    where we use \eqref{eq:equalityOfDist} to get the second inequality and \eqref{eq:estimateOfO} to get the third one. This proves \eqref{eq:relationBetweenMetrics1}. 

To prove \eqref{eq:relationBetweenMetrics2}, we may assume $\dist(\Phi(\xi,t), \Phi(\eta,s) )\geq r$ and $\dist_t(\Phi(\xi,t), \Phi(\eta,s) )< r$. Let $\gamma:[0,1] \rightarrow M$ be a geodesic in the metric $g_t\oplus d\tau^2$. By Claim \ref{lem:geodesiceInParametrization}, we have $\gamma([0,1]) \subset \Phi(\mathbb{B}^m_{\frac{r}{\sqrt{\tilde{c}}}}(\xi,t)\cap \mathbb{H}^m)$ with $(\tilde{\alpha},\tilde{\gamma}_m) = \tilde{\gamma} = \Phi^{-1}\circ \gamma$ satisfying $L(\tilde{\gamma})< \frac{r}{\sqrt{\tilde{c}}}$. The same estimates as before are satisfied, hence we conclude:
\begin{align}
        r 
    &< 
        \dist(\Phi(\xi,t),\Phi(\eta,s) )
    \leq
        \int_0^1
        g^\frac{1}{2}_{\gamma(w)}
        (\dot{\gamma}(w),\dot{\gamma}(w))\di w\\
    &\leq
        \int_0^1
        (\eta_t)^\frac{1}{2}_{\gamma(w)}
        (\dot{\gamma}(w),\dot{\gamma}(w))\di w
    +
    \left|
    \int_0^1
        O(\tilde{\alpha}(w), \tilde{\gamma}_m(w) ,t,\dot{\tilde{\gamma}}(w))
        \di w
    \right|\\
    &\leq
        \dist_t(\Phi(\xi,t),\Phi(\eta,s) )
    +
        \frac{K}{\tilde{c}}r^2.
\end{align}
\end{proof}

We omit the proof of the next elementary claim.

\begin{claim}
\label{lem:geodesicsOfCylinder}
    Let $(x,t), (y,\tau) \in \partial M \times [0,\epsilon]$ and $s \in [0,\epsilon]$. Let $\gamma:[0,1] \rightarrow \partial M$ be the geodesic between $x$ and $y$ in the metric $g_s$. Then the curve
    \begin{equation}
        w \in [0,1] \mapsto (\gamma(w), (1-w)t + w \tau)
    \end{equation}
    is a geodesic in the metric $g_s\oplus d\tau^2$.
\end{claim}

\paragraph{\bf Step 3.}
For every $t \in [0,\epsilon]$, consider the  exponential map $\exp^{g_t}$ given by the metric $g_t$. Since $g_t$ varies smoothly with respect to $t$, and $\partial M \times [0,\epsilon]$ is compact,  we know that there exists $\delta>0$ lower than the injectivity radius of each $\exp^{g_t}$. For any $(\xi, t,\zeta,s)$ in
\[
D_\delta := 
    [-1,1]^{m-1}\times [0,\epsilon] \times \mathbb{B}^{m-1}_\delta(0) \times [0,\epsilon]
\]
define
    \begin{equation}
        \Psi(\xi,t,\zeta,s)
    \df
        E(\exp^{g_t}_{\psi(\xi)}((d\psi)_\xi A_{(\xi,t)}^{-\frac{1}{2}}\zeta),s).
    \end{equation}
We may write
\[
        \Psi_{(\xi,t)}(\zeta,s)
    = \Psi(\xi,t,\zeta,s)
\]
to see $\Psi$ as a function of the two last variables only, the two first being frozen.  Observe that for any $(\xi, t,\zeta)$ in $[-1,1]^{m-1}\times [0,\epsilon] \times \mathbb{B}^{m-1}_\delta(0)$,
    \begin{align}
        (g_t)^\frac{1}{2}_{\psi(\xi)}((d\psi)_\xi A_{(\xi,t)}^{-\frac{1}{2}}\zeta,(d\psi)_{\xi} A_{(\xi,t)}^{-\frac{1}{2}}\zeta)
    &=
        (\psi^* g_t)_\xi^\frac{1}{2}(A_{(\xi,t)}^{-\frac{1}{2}}\zeta,A_{(\xi,t)}^{-\frac{1}{2}}\zeta)\\
    &=
        \left|
            \zeta^T
                A_{(\xi,t)}^{-\frac{1}{2}}
                A_{(\xi,t)}
                A_{(\xi,t)}^{-\frac{1}{2}}
            \zeta
        \right|^\frac{1}{2} \qquad \text{by  \eqref{eq:pullBackOfBoundaryMetrics}}\\
    &=
        |\zeta|
    <
        \delta.
    \end{align}
    Since $\delta$ is lower than the injectivity radius of the exponentials,  the map
    \begin{equation}
        \zeta \in \mathbb{B}^{m-1}_\delta(0)
    \mapsto
        \exp^{g_t}_{\psi(\xi)}((d\psi)_\xi A_{(\xi,t)}^{-\frac{1}{2}}\zeta)
    \end{equation}
    is injective. Thus for every $(\xi,t) \in [-1,1]^{m-1}\times[0,\epsilon]$ the map $\Psi_{(\xi,t)}$ defined on $\mathbb{B}^{m-1}_\delta(0)\times[0,\epsilon]$ is a local parametrization of $M$.  Moreover,
    \begin{equation}\label{eq:PsiPhi}
        \Psi_{(\xi,t)}(0,t)
    =
        E(\psi(\xi),t)
    =
        \Phi(\xi,t),
    \end{equation}
    and
    \begin{equation}
        \det\left(\left[\Psi_{(\xi,t)}^*g\right]_{(0,t)} \right)
    =
        1.
    \end{equation}

\begin{claim}
\label{lem:CylinderDistance}
    Consider $(\xi,t) \in [-1,1]^{m-1}\times [0,\epsilon]$, and $(\zeta,s) \in \mathbb{B}^{m-1}_\delta(0) \times [0,\epsilon]$. Then
    \begin{equation}
        \dist_t(\Psi_{(\xi,t)}(\zeta,s), \Psi_{(\xi,t)}(0,t))
    =
        \sqrt{|\zeta|^2+(t-s)^2}.
    \end{equation}
\end{claim}
\begin{proof}
    By Claim \ref{lem:geodesicsOfCylinder}, we know that the geodesic between $E^{-1}(\Psi_{(\xi,t)}(0,t))$ and $E^{-1}(\Psi_{(\xi,t)}(\zeta,s))$ in the metric $g_s \oplus d \tau^2$ is 
\begin{align*}
\gamma  : w \in [0,1] & \mapsto  (\exp^{g_t}_{\psi(\xi)}((d\psi)_\xi A_{(\xi,t)}^{-\frac{1}{2}}w\zeta), (1-w)t+ws)\\
&  \fd (\tilde{\gamma}(w),\gamma_m(w)).
\end{align*}    
     Then we also know that, for any $w \in [0,1]$,
    \begin{equation}
        (g_t)_{\tilde{\gamma}(w)}^\frac{1}{2}(\dot{\tilde{\gamma}}(w), \dot{\tilde{\gamma}}(w)) = |\zeta|.
    \end{equation}
    Thus
    \begin{align}
        \dist_t(\Psi_{(\xi,t)}(\zeta,s), \Psi_{(\xi,t)}(0,t))
    &=
        \int_0^1
        \sqrt{
            (g_t)_{\tilde{\gamma}(w)}^2(\dot{\tilde{\gamma}}(w),\dot{\tilde{\gamma}}(w))
        +
            (s-t)^2
        }     
        \di w\\
    &=
        \sqrt{|\zeta|^2 + (s-t)^2}.
    \end{align}
\end{proof}

\begin{claim}
\label{lem:DifferenceOfBallsLemma}
    There exist $r_0,\kappa>0$ such that for all $(\xi,t) \in [-2,2]^{m-1}\times[0,\epsilon/2]$ and $r \in (0,r_0)$ such that $\mathbb{B}^m_{r/\sqrt{\tilde{c}}}(\xi,t) \subset [-3,3]^{m-1}\times[0,\epsilon]$, we have
    \begin{equation}
        \mathcal{L}^m\left(
            \mathbb{B}^m_r(0,t) \, 
        \triangle \, 
            \bigg(\Psi_{(\xi,t)}^{-1}\left(B_r(\Psi_{(\xi,t)}(0,t))\right)
        \right) \bigg)
    \leq
        \kappa r^{m+1}.
    \end{equation}
\end{claim}
\begin{proof}
For any $(\xi,t) \in [-2,2]^{m-1}\times[0,\epsilon/2]$, there exists $r_0(\xi,t)>0$ small enough such that
    \begin{equation}\label{eq:inclusion_of_balls}
        B_{r_0}(\Psi_{(\xi,t)}(0,t))
    =
        B_{r_0}(\Phi(\xi,t))
    \subset 
        \Psi_{(\xi,t)}(\mathbb{B}^{m-1}_{\delta}(0) \times [0,\epsilon]).
    \end{equation}
    By compactness of $[-2,2]^{m-1}\times[0,\epsilon/2]$ and continuity of the maps $\Psi_{(\xi,t)}^{-1}$, we get that there exists a common $r_0>0$ such that the previous holds for any $(\xi,t) \in [-2,2]^{m-1}\times[0,\epsilon/2]$. Consider $r < r_0$ and $(\xi,t) \in [-2,2]^{m-1}\times[0,\epsilon/2]$, then 
    \begin{equation}
        B_r(\Psi_{(\xi,t)}(0,t)) \subset \text{Im}(\Psi_{(\xi,t)}).
    \end{equation}
    Set $A_1
    :=
        \Psi_{(\xi,t)}^{-1}(B_r(\Psi_{(\xi,t)}(0,t)))$ and
        $A_2
    :=
        \mathbb{B}^m_r(0,t).$    We will show that there exists $K>0$ such that \begin{equation}\label{eq:toprove}
            A_1\backslash A_2 \subset \mathbb{B}^m_{r+K r^2}(0,t)\backslash\mathbb{B}^m_r(0,t).
            \end{equation}
            For $(\zeta,s) \in A_1\backslash A_2$,  we know that
    \begin{equation}
        \dist(\Psi_{(\xi,t)}(\zeta,s),\Psi_{(\xi,t)}(0,t))
    <
        r.
    \end{equation}
    Therefore, from \eqref{eq:PsiPhi} and Claim \ref{lem:relationBetweenMetrics}, we conclude that there exists $K>0$ such that
    \begin{equation}
        \dist_t(\Psi_{(\xi,t)}(\zeta,s),\Psi_{(\xi,t)}(0,t))
    <
        r+Kr^2.
    \end{equation}
    Then we get from Claim \ref{lem:CylinderDistance} that
    \begin{equation}
        \sqrt{|\zeta|^2 + (t-s)^2}
    <
        r+Kr^2
    \end{equation}
    hence \eqref{eq:toprove} is proved. A similar proof shows that $$A_2\backslash A_1 \subset \mathbb{B}^m_r(0,t)\backslash \mathbb{B}^m_{r-Kr^2}(0,t).$$ From the latter and \eqref{eq:toprove} we conclude that
    \begin{equation}
        A_1\triangle A_2
    \subset
        \mathbb{B}^m_{r+Kr^2}(0,t)
    \backslash
        \mathbb{B}^m_{r-Kr^2}(0,t).
    \end{equation}
    Thus
    \begin{align}
        \mathcal{L}^m(A_1 \triangle A_2)
    &\leq
        \mathcal{L}^m(\mathbb{B}^m_{r+Kr^2}(0,t)
    \backslash
        \mathbb{B}^m_{r-Kr^2}(0,t))\\
    &\leq
        r^m \mathcal{L}^m(\mathbb{B}^m_{1+Kr}(0)\backslash \mathbb{B}^m_{1-Kr}(0))\\
    & =     r^m \omega_m ((1+Kr)^m - (1-Kr)^m)\\
    &\leq
        r^mC(Kr) \qquad \qquad \qquad \qquad \qquad \qquad \text{for some $C>0$}\\
    & \leq
        (CK)r^{m+1}.
    \end{align}
\end{proof}

\begin{claim}
\label{lem:quotientDIfferenceOfBalls}
    There exist $C,r_0>0$ such that for all $r \in (0,r_0)$ and $$(\xi,t,\zeta,s) \in [-1,1]^{m-1}\times [0,\epsilon/4] \times \mathbb{B}^{m-1}_\delta(0) \times [0,\epsilon]$$ such that $\mathbb{B}^m_{2r/\sqrt{\tilde{c}}}(\xi,t) \cap \mathbb{H}^m \subset [-2,2]^{m-1}\times[0,\epsilon/2]$ and $\dist(\Psi_{(\xi,t)}(\zeta,s), \Phi(\xi,t))<r$, then
    \begin{equation}
        \left|
            \frac{1}{V(\Psi_{(\xi,t)}(\zeta,s),r)}
        -
            \frac{1}{\mathcal{L}^m(\mathbb{B}^m_r(0,s) \cap \mathbb{H}^m)}
        \right|
    \leq
        \frac{C}{r^{m-1}} \, \cdot
    \end{equation}
\end{claim}
\begin{proof} Let us first consider the map $G:D_{\delta} \rightarrow \mathbb{R}$ given by
    \begin{equation}
        G(\xi,t,\zeta,s)
    \df
        \det\left(\left[\Psi_{(\xi,t)}^*g\right]_{(\zeta,s)}\right)^\frac{1}{2}.
    \end{equation}
    This map is $C^\infty$, and its value at any $(\xi,t)\times(0,t)$ is $1$, thus by a Taylor expansion in the variable $(\zeta,s)$ centered at $(0,t)$, and compactness of $D_\delta$, we obtain that there exists $k>0$ such that for all $(\xi,t,\zeta,s) \in D_{\delta}$ we have
    \begin{equation}\label{eq:G}
        \left|
            \det\left(\left[\Psi_{(\xi,t)}^*g\right]_{(\zeta,s)}\right)^\frac{1}{2}
        -
            1
        \right|
    \leq
        k|(\zeta,s)-(0,t)|.
    \end{equation} 
    In particular there exists $C>0$ such that for all $(\xi,t,\zeta,s) \in D_{\delta}$,
    \begin{equation}
        \left|
            \det\left(\left[\Psi_{(\xi,t)}^*g\right]_{(\zeta,s)}\right)^\frac{1}{2}
        \right|
    \leq
        C.
    \end{equation}

Let us now consider $r,\xi,t,\zeta,s$ as in the statement of the claim. Set $y \df \Psi_{(\xi,t)}(\zeta,s)$. Since $\dist(y, \Phi(\xi,t))< r$, and $\mathbb{B}^m_{2r/\sqrt{\tilde{c}}}(\xi, t) \subset [-2,2]^{m-1} \times [0,\epsilon/2]$, we know by Claim \ref{lem:geodesiceInParametrization} that $y \in \Phi(\mathbb{B}^m_{r/\sqrt{\tilde{c}}}(\xi,t) \cap \mathbb{H}^m)$. Thus if we set
    \begin{equation}
    \label{eq:inverseOfPointPhi}
        (\eta, \tau) \df \Phi^{-1}(y) \in [-2,2]^{m-1}\times[0,\epsilon/2],
    \end{equation}
    we obtain that
    \begin{equation}
        \mathbb{B}^m_{r/\sqrt{\tilde{c}}}(\eta, \tau)
        \cap
        \mathbb{H}^m
    \subset
        \mathbb{B}^m_{2r/\sqrt{\tilde{c}}}(\xi,t)
        \cap
        \mathbb{H}^m
    \subset 
        [-2,2]^{m-1}\times[0,\epsilon/2].
    \end{equation}
    Thus by Claim \ref{lem:geodesiceInParametrization}, we conclude that $B_r(y) = B_r(\Phi(\eta, \tau)) \subset \Phi(\mathbb{B}^m_{r/\sqrt{\tilde{c}}}(\eta, \tau))$.

    By \eqref{eq:PsiPhi} and \eqref{eq:defPhi}, we easily see that 
    \begin{equation}
    \label{eq:equalityOfHeight}
        s = \tau.
    \end{equation}
     Moreover, by \eqref{eq:inverseOfPointPhi} and \eqref{eq:equalityOfHeight}, we also have that
    \begin{equation}
        \Psi_{(\eta,s)}(0,s)
    =
        E(\psi(\eta), s)
    =
        \Phi(\eta,s)
    =
        y.
    \end{equation}
    Choose $r_0$ such that $\frac{r_0}{\sqrt{\tilde{c}}} < \frac{\epsilon}{4}$. Since $t \leq \frac{\epsilon}{4}$ and $(\eta,s) \in \mathbb{B}^m_{r/\sqrt{\tilde{c}}}(\xi,t)$, this implies that $s \leq \frac{\epsilon}{2}$. Thus we can use Claim \ref{lem:DifferenceOfBallsLemma} to ensure that
    \begin{equation}
    \label{eq:ballLemmaEqForY}
        \mathcal{L}^m\left(
            \mathbb{B}^m_r(0,s)
        \triangle
            \left(\Psi_{(\eta,s)}\right)^{-1}\left(B_r(\Psi_{(\eta,s)}(0,s))\right)
        \right)
    \leq
        \kappa r^{m+1}.
    \end{equation}
   Then
    \begin{align}
       V(y,r)
    &=
        \int_{\Psi_{(\eta, z)}^{-1}(B_r(y))}
        \left|\det\left[ \Psi_{(\eta,s)}^*g\right]_w\right|^\frac{1}{2}
        \di\mathcal{L}^m(w)\\
    &=
        \int_{\Psi_{(\eta, s)}^{-1}(B_r(y))}
        1\cdot\di\mathcal{L}^m(w)
    +
        O(r^{m+1}) \qquad \quad  \quad \, \text{by \eqref{eq:G}}\\
    &=
        \mathcal{L}^m(\mathbb{B}^m_r(0,s)\cap\mathbb{H}^m)
    +
        O(r^{m+1})  \qquad \qquad \qquad  \text{by \eqref{eq:ballLemmaEqForY}}
    \end{align}
    that is, there exists $\tilde{C}>0$ such that
    \begin{equation}
    \left|
        \vol_g(B_r({y}))
    -
        \mathcal{L}^m(\mathbb{B}^m_r(0,z)\cap \mathbb{H}^m)
    \right|
    \leq    
    \tilde{C}r^{m+1}
    \end{equation}
    Thus using the local Ahlfors regularity of $(M,g)$ and Claim \ref{lem:DifferenceOfBallsLemma}, we obtain
    \begin{align}
        \left|
            \frac{1}{\vol_g(B_r(\overline{y}))}
        -
            \frac{1}
            {\mathcal{L}^m(\mathbb{B}^m_r(0,z)\cap \mathbb{H}^m)}
        \right|
    &=
        \left|
            \frac{\mathcal{L}^m(\mathbb{B}^m_r(0,z)\cap \mathbb{H}^m) - \vol_g(B_r(\overline{y}))}{\vol_g(B_r(\overline{y}))\mathcal{L}^m(\mathbb{B}^m_r(0,z)\cap \mathbb{H}^m)}
        \right|\\
    &\leq
        \frac{C \tilde{C} r^{m+1}}{ r^{m} \mathcal{L}^m(\mathbb{B}^m_r(0,z)\cap \mathbb{H}^m)}\\
    &
        \leq
            \frac{C \tilde{C}}{\mathcal{L}^m(\mathbb{B}^m_1(0)\cap \mathbb{H}^m) r^{m-1}}
    \end{align}
    concluding the proof.
\end{proof}

\paragraph{\bf Step 4.}

\begin{claim}
\label{lem:TaylorExpansionInNormalizedCoords}
    For all $f \in C^\infty(M)$, there exists $C>0$ such that for all $(\xi,t, \zeta, s) \in D_{\delta}$ we have
    \begin{equation}
        \left|
            f\circ \Psi_{(\xi,t)}(\zeta,s)
        -
            f\circ \Psi_{(\xi,t)}(0,t)
        \right|
    \leq
        C\|
            (\zeta,s)
        -
            (0,t)
        \|_2,
    \end{equation}  
    \begin{equation}
        \left|
            f\circ \Psi_{(\xi,t)}(\zeta,s)
        -
            f\circ \Psi_{(\xi,t)}(0,t)
        -
            \nabla(f\circ \Psi_{(\xi,t)})_{(0,t)}
            \cdot
            \left(
                (\zeta, s)-(0,t)
            \right)
        \right|
    \leq
        C\|
            (\zeta, s)
        -
            (0,t)
        \|_2^2
    \end{equation}
    Also if $\partial_\nu f|_{\partial M} = 0$, then there exists $C>0$ such that
    \begin{equation}
    \label{eq:zeroNormalDerivativeBound}
        \left|
            \partial_{m}
            f\circ\Psi_{(\xi,t)}
            (0,t)
        \right|
    \leq
        Ct
    \end{equation}  
\end{claim}
\begin{proof}
    The map $\tilde{f}((\xi,t),(\zeta, s)) = f\circ \Psi_{(\xi,t)}(\zeta, s)$ is $C^\infty$, thus by a Taylor expansion of order $1$ and $2$ respectively, and compactness of $D_{\delta}$, we conclude the first two inequalities. For the last, we notice that
    \begin{equation}
        \partial_{m}
            f\circ\Psi_{(\xi,0)}
            (0,0)
        =
        (\partial_\nu f)(\psi(\xi,0))
        =
        0.
    \end{equation}
    Thus by a Taylor expansion and compactness of $D_{\delta}$ we conclude \eqref{eq:zeroNormalDerivativeBound}.
\end{proof}

 Now we will fix a function $f \in C^\infty(M)$ such that $\partial_\nu f|_{\partial M} = 0$, and show that for $x \in M$ such that $d(x,\partial M) < \frac{\epsilon}{4}$ then $\Delta_r f(x)$ is uniformly bounded. The proof for points $x$ with $d(x,\partial M) \geq\frac{\epsilon}{4}$, follows from the uniform convergence obtained in Proposition \ref{lem:LaplaceBeltramiLemma}. We will study the following term of the AMV:
\begin{equation}
    G_r(x)
:=
    \frac{1}{r^2}
    \int_{B_r(x)}
    \frac{1}{\vol_g(B_r(y))}
    \left(
        f(y)
        -
        f(x)
    \right)
    \di \vol_g(y)
\end{equation}
since the bound for the remainder follows similarly.

We notice that by equation \eqref{eq:coverOfBoundary} we have that there exists some $i \in \{1,...,l\}$ and $(\xi,t) \in [-1,1]^{m-1}\times[0,\frac{\epsilon}{4}]$ such that $x = \Phi_i(\xi,t)$. We let $\Phi = \Phi_i$. Also by \eqref{eq:inclusion_of_balls}, for $0<r<r_0$ in the conditions of the Claim, we have that $B_r(x) = B_r(\Psi_{(\xi,t)}(0,t)) \subset \Psi_{(\xi,t)}(\mathbb{B}^{m-1}_\delta(0)\times[0,\epsilon])$. Also we can choose $r_0$ small enough so that for all $(\xi,t) \in [-1,1]^{m-1}\times[0,{\epsilon}/{4}]$ we have $\mathbb{B}_{2r_0/\sqrt{\tilde{c}}}(x,t) \cap \mathbb{H}^m \subset [-2,2]^{m-1}\times [0,\epsilon/2]$. With this we can apply Claim \ref{lem:geodesiceInParametrization} to conclude that $B_r(\Psi_{(\xi,t)(0,t)}) \subset \Phi(\mathbb{B}^m_{r/\sqrt{\tilde{c}}}(\xi,t))$ and we can also apply Claim \ref{lem:quotientDIfferenceOfBalls} for points $\Psi_{(\xi,t)}(\zeta,s) \in B_r(\Psi_{(\xi,t)}(0,t))$.

Thus we can change variables of the integral to obtain
\begin{align}
    &G_r(x)
=
    \frac{1}{r^2}
    \int_{\Psi_{(\xi,t)}^{-1}(B_r(\Psi_{(\xi,t)}(0,t)))}
    \frac{1}{\vol_g(B_r(\Psi_{(\xi,t)}(\zeta,s) ))}\\
    &\times\left(
        f(\Psi_{(\xi,t)}(\zeta,s))
        -
        f(\Psi_{(\xi,t)}(0,t))
    \right)
    \det\left(\left[\Psi_{(\xi,t)}^*g)\right]_{\zeta,s}\right)^\frac{1}{2}
    \di\mathcal{L}^m(\zeta,s)
\end{align}

\begin{align}
    \begin{split}
G_r(x)
    &= 
    \frac{1}{r^2}
    \int_{\Psi_{(\xi,t)}^{-1}(B_r(\Psi_{(\xi,t)}(0,t)))}
        \frac
        {
            f(\Psi_{(\xi,t)}(\zeta,s))
            -
            f(\Psi_{(\xi,t)}(0,t))
        }
        {
            \mathcal{L}^m(\mathbb{B}^m_r(0,s)\cap \mathbb{H}^m)
        }
    \times\\
    &\quad\times 
        \det\left(\left[\Psi_{(\xi,t)}^*g)\right]_{\zeta,s}\right)^\frac{1}{2}
        \ \di\mathcal{L}^m(\zeta,s) + O(1)
\end{split}
\tag{by \ref{lem:quotientDIfferenceOfBalls} and 
        \ref{lem:TaylorExpansionInNormalizedCoords}}
\\
\begin{split}
    &= 
    \frac{1}{r^2}
    \int_{\mathbb{B}^m_r(0,t) \cap \mathbb{H}^m}
        \frac
        {
            f(\Psi_{(\xi,t)}(\zeta,s))
            -
            f(\Psi_{(\xi,t)}(0,t))
        }
        {
            \mathcal{L}^m(\mathbb{B}^m_r(0,s)\cap \mathbb{H}^m)
        }
    \times\\
    &\quad\times 
        \det\left(\left[\Psi_{(\xi,t)}^*g)\right]_{\zeta,s}\right)^\frac{1}{2}
        \ \di\mathcal{L}^m(\zeta,s) + O(1)
\end{split}
\tag{by \ref{lem:DifferenceOfBallsLemma} and 
        \ref{lem:TaylorExpansionInNormalizedCoords}}
\\
\begin{split}
    &= 
    \frac{1}{r^2}
    \int_{\mathbb{B}^m_r(0,t) \cap \mathbb{H}^m}
        \frac
        {
            f(\Psi_{(\xi,t)}(\zeta,s))
            -
            f(\Psi_{(\xi,t)}(0,t))
        }
        {
            \mathcal{L}^m(\mathbb{B}^m_r(0,s)\cap \mathbb{H}^m)
        }
    \times\\
    &\quad\times 
        \ \di\mathcal{L}^m(\zeta,s) + O(1)
\end{split}
\tag{by \ref{lem:determinantLemma} and 
        \ref{lem:TaylorExpansionInNormalizedCoords}}
\\
\begin{split}
    &= 
    \frac{1}{r^2}
    \int_{\mathbb{B}^m_r(0,t) \cap \mathbb{H}^m}
        \frac
        {
            \sum_{i=1}^{m-1}
            \partial_j(f\circ\Psi_{(\xi,t)})(0,t)
            \zeta_i
            +
            \partial_m(f\circ \Psi_{(\xi,t)})(0,t)
            (s-t)
        }
        {
            \mathcal{L}^m(\mathbb{B}^m_r(0,s)\cap \mathbb{H}^m)
        }
    \times\\
    &\quad\times 
        \ \di\mathcal{L}^m(\zeta,s) + O(1)
\end{split}
\tag{by \ref{lem:TaylorExpansionInNormalizedCoords}}
\\
\begin{split}
    &= 
    \frac{1}{r^2}
    \int_{\mathbb{B}^m_r(0,t) \cap \mathbb{H}^m}
        \frac
        {
            \sum_{i=1}^{m-1}
            \partial_j(f\circ\Psi_{(\xi,t)})(0,t)
            \zeta_i
            +
            \partial_m(f\circ \Psi_{(\xi,t)})(0,t)
            (s-t)
        }
        {
            \mathcal{L}^m(\mathbb{B}^m_r(0,s)\cap \mathbb{H}^m)
        }
    \times\\
    &\quad\times 
        \ \di\mathcal{L}^{m-1}(\zeta)\di\mathcal{L}(s)) + O(1)
\end{split}
\tag{by \ref{lem:TaylorExpansionInNormalizedCoords}}
\\
\begin{split}
    &= 
    \frac{1}{r^2}
    \int_{\mathbb{B}^m_r(0,t) \cap \mathbb{H}^m}
        \frac
        {
            \partial_m(f\circ \Psi_{(\xi,t)})(0,t)
            (s-t)
        }
        {
            \mathcal{L}^m(\mathbb{B}^m_r(0,s)\cap \mathbb{H}^m)
        }
    \times\\
    &\quad\times 
        \ \di\mathcal{L}^{m-1}(\zeta)\di\mathcal{L}(s)) + O(1)
\end{split}
\tag{by symmetry}
\\
\end{align}

Now we separate further in two cases. First if $\dist(x,\partial M) > 2r$ then $t > 2r$ and so for all $(\zeta,s) \in \mathbb{B}^m_r(0,t) \cap \mathbb{H}^m$ we have $\mathcal{L}^m(\mathbb{B}^m_r(0,s) \cap \mathbb{H}^m) = \mathcal{L}^m(\mathbb{B}^m_r(0))$, and so we conclude that
\begin{align}
    &\frac{1}{r^2}
    \int_{\mathbb{B}^m_r(0,t) \cap \mathbb{H}^m}
        \frac
        {
            \partial_m(f\circ \Psi_{(\xi,t)})(0,t)
            (s-t)
        }
        {
            \mathcal{L}^m(\mathbb{B}^m_r(0,s)\cap \mathbb{H}^m)
        }
        \ \di\mathcal{L}^{m-1}(\zeta)\di\mathcal{L}(s)\\
    &=
        \frac{1}{r^2\mathcal{L}^m(\mathbb{B}^m_r(0))}
    \int_{\mathbb{B}^m_r(0,t) \cap \mathbb{H}^m}
            \partial_m(f\circ \Psi_{(\xi,t)})(0,t)
            (s-t)
        \ \di\mathcal{L}^{m-1}(\zeta)\di\mathcal{L}(s)\\
    &=
        0.
\end{align}
This shows that $G_r(x) = O(1)$ for $x$ such that $2r < d(x,\partial M)<\frac{\epsilon}{4}$ since there are only a finite number of parametrizations. On the other hand if $d(x,\partial M) \leq 2r$, then we have by Lemma \ref{lem:TaylorExpansionInNormalizedCoords} that for $t \leq 2r$ then $\left|\partial_m(f\circ \Psi_{(x,t)})(0,t)\right| \leq 2Cr$, and so 
\begin{equation}
    \frac{1}{r^2}
    \int_{\mathbb{B}^m_r(0,t) \cap \mathbb{H}^m}
        \frac
        {
            \partial_m(f\circ \Psi_{(x,t)})(0,t)
            (s-t)
        }
        {
            \mathcal{L}^m(\mathbb{B}^m_r(0,s)\cap \mathbb{H}^m)
        }
        \ \di\mathcal{L}^{m-1}(y)\di\mathcal{L}(s)
    =
    O(1),
\end{equation}
which shows that $G_r(x) = O(1)$ for $x$ such that $d(x, \partial M) < 2r$.
\end{proof}

\section{Spectral convergence}\label{sec:manifolds}

In this section, we prove Theorem \ref{th:3}. We consider a smooth, compact, connected manifold $M^m$ endowed with a smooth Riemannian metric $g$. We let $\dist_g$ and $\vol_g$ be the associated Riemannian distance and volume measure on $M$, respectively. If $\partial M = \emptyset$ (resp.~$\partial M \neq \emptyset$), we let $\{\mu_k\}_{k \in \mathbb{N}}$ be the sequence of Laplace (resp.~Neumann) eigenvalues of $(M,g)$.

\subsection{Existence of limit eigenfunctions} Recalling that $C^\infty_{\nu}(M)$ is defined in \eqref{eq:Cnu}, we define the Hilbert space
\[
H \df \overline{C^\infty_{\nu}(M)}^{\|\cdot\|_{W^{2,2}}}.
\]
We let $\Pi(M, \vol_g)$ be defined as in \eqref{eq:perp} and we consider the operator $T: \Pi(M, \vol_g) \rightarrow H^*$ which maps any $f \in \Pi(M, \vol_g)$ to
\begin{equation}
\label{OperatorL2ToFuncs}
    T(f)
\df
   \left(
   v \in H
   \mapsto
   -
    \int_M
    f \Delta_g v \di \vol_g
    \right).
\end{equation}

\begin{lemma}
\label{injectivityOfIntegralOfLaplacian}
    The operator $T$ is injective.
\end{lemma}
\begin{proof}
    For $f \in \Pi(M, \vol_g)$ such that $f \neq 0$, let $v \in H$ be the solution of
    \begin{equation}
        \begin{cases}
            -{\Delta}_g v = f & \text{in } M,\\
            \partial_\nu v = 0 & \text{in } \partial M.
        \end{cases}
    \end{equation}  
    The solution to this problem exists since $\int_M f = 0 = \int_{\partial M} \partial_\nu v$. In fact by regularity theory we can conclude that $v \in H$, and so we conclude that
    \begin{equation}
        T(f)(v)
    =
        -\int_M f{\Delta}_g v\di \vol_g
    =
        \int_M f^2\di \vol_g
    \neq 0.
    \end{equation}
    Thus $T(f) \neq 0$, concluding that $T$ is injective.
\end{proof}

Let us now prove the existence of $L^2$-weak limit eigenfunctions.

\begin{prop}
\label{weakConvergenceLemma}
     Let $(r_n)$ be a sequence of positive numbers such that $r_n \to 0$. For any $n$ let $(\lambda_{k,{r_n}})$ be the eigenvalues of the operator $\tilde{\Delta}_{r_n}$ and let $(f_{k,r_n})$ be corresponding eigenfunctions. Then for any $k$ there exists a Laplace (resp.~Neumann) eigenfunction $f$ of $(M,g)$ with associated eigenvalue $\mu$ such that, up to extracting subsequences, satisfy
    \begin{equation}\label{eq:conv1}
    f_{k,r_n} \stackrel{L^2}{\rightharpoonup} f,
    \end{equation}
    \begin{equation}\label{eq:conv2}
    \lambda_{k,r_n} \to C_m \,  \mu,
    \end{equation}
    \begin{equation}\label{eq:energy_bound}
    \sup_n E_{r_n}(f_{k,r_n}-f) < +\infty.
    \end{equation}
\end{prop}
    
\begin{proof} By the proof of Theorem \ref{th:1}, in particular \eqref{eq:boundOfEigenvalues}, there exists $\lambda \geq 0$ and a subsequence such that $\lambda_k(-\tilde{\Delta}_{r_n}) \rightarrow \lambda$ up to  subsequence. Since $\|f_{k,r_n}\|_{L^2(M)} = 1$ for any $n$, there exists $f \in L^2(M,\vol_g)$ such that the weak convergence \eqref{eq:conv1} holds up to subsequence. Therefore, by Theorem \ref{th:2}, we get that for any $\psi \in C^\infty(M)$ (resp.~$C^\infty_{\partial_{\nu}}(M)$),
    \begin{align*}
    \int_{M} f \Delta_g \psi \di \vol_g
    &=
    \lim_n
    \frac{1}{C_m}
    \int_M
        f_{k,r_n}
        \tilde{\Delta}_{r_n}\psi \di \vol_g
    =
        \frac{1}{C_m}
        \lim_{n}
        \int_M
            \left( \tilde{\Delta}_{r_n}f_{k,r_n} \right)\psi \di \vol_g\\
    &=
        -\frac{1}{C_m}
        \lim_{n}
        \lambda_{k}(-\tilde{\Delta}_{r_n})
        \int_M
            f_{k,r_n}\psi
            \di \vol_g
    =
        -\frac{1}{C_m}
        \lambda
        \int_M
            f
            \psi
            \di \vol_g.
    \end{align*}
    Moreover, since for any $n$ it holds that $\lambda_{k,r_n} > 0$ and $$0=\int_M -\tilde{\Delta}_{r_n} f_{k,r_n} \di \vol_g = \int_M \tilde{\lambda}_{k,r_n}f_{k,r_n} \di \vol_g$$ we get that $\int_{M} f_{k,r_n} = 0$. Thus $\int_M f \di \vol_g= 0$ by weak convergence.

    Now let $v \in W^{2,2}(M)$ be the solution of
    \begin{equation}
    \begin{cases}
  -{\Delta}_g v = \frac{\lambda}{C_m}f  & \text{ in } M \\
    \partial_\nu v = 0 & \text{ in } \partial M
    \end{cases}
    \end{equation}
    satisfying $\int_M v = 0$. Then we have that for $\psi \in H \cap C^\infty(M)$ then
    \begin{equation}
        \int_M v{\Delta}_g \psi
        \di \vol_g
    =
        -\frac{1}{C_m}
        \lambda
        \int_M
            f
            \psi
        \di \vol_g
    =
        \int_{M} f{\Delta}_g \psi\di \vol_g.
    \end{equation}
    Since this is a dense subspace of $H$ and the functionals are continuous with respect to $W^{2,2}(M)$ in $\psi$, the equality holds for all $H$. Thus by Lemma \ref{injectivityOfIntegralOfLaplacian} we conclude that $v = f$, and so $f$ satisfies
    \begin{equation}
    \begin{cases}
  -{\Delta}_g f = \frac{\lambda}{C_m}f  & \text{ in } M \\
    \partial_\nu f = 0 & \text{ in } \partial M
    \end{cases}
    \end{equation}
    thus $f$ is a Neumann eigenfunction, and so it must be $C^\infty(M)$. Also since both $f_{k,r_n},f \in \Pi(M,\vol_g)$, we know by Proposition \ref{prop:kernelOfSymAMV} using triangle inequality of the inner product,
    \begin{equation}
        E_{r_n}(f_{k,r_n}-f)^\frac{1}{2}
    \leq
        E_{r_n}(f_{k,r_n})^\frac{1}{2}
    +
        E_{r_n}(f)^\frac{1}{2}
    .
    \end{equation}
    We know that $E_{r_n}(f_{k,r_n}) = \lambda_{k,r_n}\|f_{k,r_n}\|_{L^2(M)} = \lambda_{k,r_n}$ which is uniformly bounded. Also since $f \in C^\infty(M)$, by Lemma \ref{convergenceOfInnerProds}, we know that $E_{r_n}(f)$ is also uniformly bounded, concluding the proof.
\end{proof}

\subsection{Energy comparison} Let us now compare the energy of a map defined on $M$ with the energy of the image of the map through a local chart parametrizing a neighborhood of an open subset of $\partial M$. To this aim, up to scaling, we consider a map $\Phi : (-1,1)^{m-1} \times [0,1) \to M$ which is a bi-Lipschitz homeomorphism onto its image.
We set
\begin{equation}\label{eq:mathcalQ}
\mathcal{Q} \df (-1/2,1/2)^{m-1} \times [0,1/2).
\end{equation}

\begin{lemma}
\label{ineqBetweenEnergies}
    There exist constants $\tilde C = \tilde C(\Phi)>0$ and $\tilde{c} = \tilde c (\Phi)>0$ such that for any $f \in L^2(M)$, for any $r \in (0,1/2)$,
    \begin{equation}
        \tilde{E}_{\tilde{c} r,\mathfrak{Q}}(f\circ \Phi)
    \leq
        \tilde{C}E_{r,\mathfrak{M}}(f).
    \end{equation}
    where $\mathfrak{Q} \df (\mathcal{Q},d_\infty,\mathcal{L}^m)$ and $\mathfrak{M} \df (M,d,\mu)$.
\end{lemma}

\begin{proof}
We start by pointing out that there exist constants $c=c(\Phi)>0$ and $C=C(\Phi)>0$ such that for all $x \in \mathcal{Q}$ and $r \in (0,1/2)$,
\begin{equation*}
    \Phi(Q_{cr}(x))
\subset
    B_r(\Phi(x))
\subset
    \Phi(Q_{Cr}(x)),
\end{equation*}
\begin{equation*}
    V(\Phi(x),r)
\leq
    C\mathcal{L}^m(Q_{cr}(x)\cap \mathcal{Q})    
,
\end{equation*}
\begin{equation*}
    \det(g_x) \geq 0,
\end{equation*}
where $g_x$ is the metric in the coordinates given by $\Phi$. Then for any $x,y \in \mathcal{Q}$,
\begin{align*}
    \tilde{a}_{r,\mathfrak{M} }(\Phi(x),\Phi(y))
&=
    1_{B_r(\Phi(x))}(\Phi(y))
        \left(
            \frac{1}{V(\Phi(x),r)}
            +
            \frac{1}{V(\Phi(y),r)}
        \right)\\
&\geq
    1_{Q_{cr}(x)}(y)
    \left(
        \frac{1}{C\mathcal{L}^m(Q_{cr}(x)\cap \mathcal{Q})} +
        \frac{1}{C\mathcal{L}^m(Q_{cr}(y)\cap \mathcal{Q})}  
    \right)\\
&=
   \frac{\tilde{a}_{cr,\mathfrak{Q}}(x,y)}{C}\, \cdot
\end{align*}
Thus
\begin{align*}
    \tilde{E}_{r,\mathfrak{M}}(f)
&\geq
    \iint_{\Phi(\mathcal{Q})^2}
            \tilde{a}_{r,\mathfrak{M}}(p,q)
            \left(
                \frac{f(p)-f(q)}{r}
            \right)^2
            \di\vol_g(q)
        \di\vol_g(p)\\
& =
    \iint_{\mathcal{Q}^2}
            \tilde{a}_{r,\mathfrak{M}}(\Phi(x),\Phi(y))
            \left(
                \frac{f(\Phi(x))-f(\Phi(y))}{r}
            \right)^2 \\
            & \qquad \qquad  \qquad \qquad \qquad \qquad \times
            \sqrt{\det(g_{x})\det(g_y)}
            \di\mathcal{L}^m(y)
        \di\mathcal{L}^m(x)\\
&     \geq
        \iint_{\mathcal{Q}^2}
        \frac{c}{C}
            \tilde{a}_{cr,\mathfrak{Q}}(x,y)
            \left(
                \frac{f(\Phi(x))-f(\Phi(y))}{r}
            \right)^2
            \di\mathcal{L}^m(y)
        \di\mathcal{L}^m(x)\\
    &=
        \frac{c}{C}
        E_{cr,\mathfrak{Q}}(f\circ \Phi).
\end{align*}
Taking $\tilde{c} = c$ and $\tilde{C} = \frac{C}{c}$, we obtain the result.
\end{proof}

\subsection{Proof of Theorem \ref{th:3}}

We are now in a position to prove Theorem \ref{th:3}. Recall the context of this result : $(r_n) \subset (0,+\infty)$ is a sequence such that $r_n \to 0$,  $(M^m,g)$ is a compact, connected, smooth Riemannian manifold with $\partial M = \emptyset$ (resp.~$\partial M \neq \emptyset$ ), $k$ is a positive integer, $\mu_k$ is the $k$-th lowest Laplace (resp.~Neumann) eigenvalue of $\Delta_g$, and $f_{k,r_n}$ is an eigenfunction of $-\tilde{\Delta}_{r_n}$ associated with the $k$-th eigenvalue $\lambda_k(-\tilde{\Delta}_{r_n})$ of this operator.

\begin{proof}
    \textbf{Step 1.}
    First we show strong $L^2$-convergence of the sequence $(f_{k,r_n})$. We proceed by contradiction. By Proposition \ref{weakConvergenceLemma}, we can assume that there exist $\alpha>0$, $f \in L^2(M,\mu)$ which is a Neumann eigenfunction, and $(r_n) \subset (0,+\infty)$ such that
    $r_n \to 0$ and
    \begin{equation}
    f_{k,r_n} \stackrel{L^2}{\rightharpoonup} f,
\quad \quad
    \|f_{k,r_n}-f\|_{L^2(M)}^2 \geq \alpha.
\end{equation}
Since $M$ is a compact manifold with boundary, up to scaling there exist finitely many bi-Lipschitz homeomorphisms $\{ \Phi_j: (-1,1)^{m-1}\times [0,1)  \rightarrow M\}_{j \in \{1,\ldots\ell\}}$ such that
\begin{equation}
    \bigcup_{j}\Phi_j ( \mathcal{Q})
=
    M,
\end{equation}
where $\mathcal{Q}$ is as in \eqref{eq:mathcalQ}, and 
\begin{equation}
\bigcup_{j} \Phi_j\Big((-1/2,1/2)^{m-1}\times \{0\} \Big) = \partial M.
\end{equation}
Then there exists $j \in \{1,\ldots,\ell\}$ such that, up to a subsequence,
\begin{equation}
   \inf_n \int_{\Phi_j(\mathcal{Q})}
        |f_{k,r_n}-f|^2 \di \vol_g
    \geq
        \frac{\alpha}{\ell}
    >
    0.
\end{equation}
From this we conclude that there exists $\tilde{\alpha} > 0$ such that
\begin{equation}\label{eq:6.4}
   \inf_n \int_{\mathcal{Q}}
        |f_{k,r_n}-f|^2 \circ \Phi_j  \di \mathcal{L}^m
    \geq
        \tilde{\alpha}
    >
    0.
\end{equation}
Let us set $\Phi:=\Phi_j$. Then there exist $C,\tilde{C},\tilde{c}>0$ such that for any $n$,
\[
\begin{array}{rlr}
C & \ge \tilde{E}_{r_n}(f_{k,r_n}-f)  & \text{by Proposition \ref{weakConvergenceLemma}} \\
   & \ge \tilde{C}^{-1} \tilde{E}_{\tilde{c}r_n,\mathfrak{Q}}((f_{k,r_n}-f)\circ \Phi) & \text{by Lemma  \ref{ineqBetweenEnergies}.}
\end{array}
\]
By the weak convergence we also have
\begin{equation}\label{eq:hnto0}
    h_n  \df (f_{k,r_n}-f)\circ \Phi
\stackrel{L^2}{\rightharpoonup}
    0.
\end{equation}
Let us set $\overline{r}_n = \tilde{c}r_n$. For an integer $N$ to be chosen later, consider a decomposition of $\mathcal{Q}$ into $L_N$ disjoint subcubes $\{\tilde{Q}_{i}\}$ of size $1/N$. For any $x,y \in \mathcal{Q}$, we set
\begin{equation}
a_r(x,y) \df \chi_{Q_r(x)\cap \mathcal{Q}}(y)\left(\frac{1}{\mathcal{L}^m(Q_r(x)\cap \mathcal{Q})}  + \frac{1}{\mathcal{L}^m(Q_r(y)\cap \mathcal{Q})}\right),
\end{equation}
\begin{equation}
    a_{r,i}(x,y) \df \chi_{Q_r(x)\cap \mathcal{Q}_i}(y)\left(\frac{1}{\mathcal{L}^m(Q_r(x)\cap \mathcal{Q}_i)}  + \frac{1}{\mathcal{L}^m(Q_r(y)\cap \mathcal{Q}_i)}\right),
\end{equation}
and we point out that for $x,y \in \mathcal{Q}_i$
\begin{equation}
    a_r(x,y)
\geq
    \frac{1}{2^m}
    a_{r,i}(x,y).
\end{equation}
We also set for any $n$,
\begin{equation}\epsilon_{i,n,N} \df \int_{\mathcal{Q}_i}h_n\di\mathcal{L}^m, \qquad \delta_{n,N} \df \max_i |\epsilon_{i,n,N}|.
\end{equation}
We obtain that for any $n$,
\begin{align}
\tilde{E}_{\overline{r}_n,\mathfrak{Q}}(h_n)
&=
\int_{\mathcal{Q}}
\left(
    \int_{\mathcal{Q}}
        a_{\overline{r}_n}(x,y)
        \frac{\left(
            h_n(x)
        -
            h_n(y)
        \right)^2}
        {\overline{r}_n^2}
    \di \mathcal{L}^m(y)
\right)
\di \mathcal{L}^m(x)\\
&=
\sum_{i}
\int_{\mathcal{Q}_i}
\left(
    \int_{\mathcal{Q}}
        a_{\overline{r}_n}(x,y)
        \frac{\left(
            h_n(x)
        -
            h_n(y)
        \right)^2}
        {\overline{r}_n^2}
    \di \mathcal{L}^m(y)
\right)
\di \mathcal{L}^m(x)\\
&\geq
    \frac{1}{2^m}
    \sum_{i}
    \int_{\mathcal{Q}_i}
\left(
    \int_{\mathcal{Q}_i}
        a_{\overline{r}_n,i}(x,y)
        \frac{\left(
            h_n(x)
        -
            h_n(y)
        \right)^2}
        {\overline{r}_n^2}
    \di \mathcal{L}^m(y)
\right)
\di \mathcal{L}^m(x)\\
&=
    \frac{1}{2^m}
    \sum_i
    \tilde{E}_{\overline{r}_n, \mathcal{Q}_i}(h_n)\\
&=
    \frac{1}{2^m}
    \sum_i
    \tilde{E}_{\overline{r}_n, \mathcal{Q}_i}(h_n - \epsilon_{i,n,N})\\
&\geq
    \frac{1}{2^m}
    \sum_i
    \|h_n-\epsilon_{i,n,N}\|_{L^2(\mathcal{Q}_i)}^2\lambda_{1}(-\tilde{\Delta}_{\overline{r}_n, \mathcal{Q}_i})\\
&\geq
    \frac{\lambda_1(-\tilde{\Delta}_{\overline{r}_n, \mathfrak{Q}^m(1/N)})}
    {2^m}
    \sum_i
    \|h_n-\epsilon_{i,n,N}\|_{L^2(\mathcal{Q}_i)}^2\\
&=
    \frac{\lambda_1(-\tilde{\Delta}_{\overline{r}_n, \mathfrak{Q}^m(1/N)})}
    {2^m}
    \sum_i
    \left(
    \|h_n\|_{L^2(\mathcal{Q}_i)}^2
-
    2\epsilon_{i,n,N}
    \int_{\mathcal{Q}_i}
        h_n
    \di\mathcal{L}^m
+
    \mathcal{L}^m(\mathcal{Q}_i)\epsilon_{i,n,N}^2
    \right)\\
&\geq
    \frac{\lambda_1(-\tilde{\Delta}_{\overline{r}_n, \mathfrak{Q}^m(1/N)})}
    {2^m}
    \left(
        \|h_n\|^2_{L^2(\mathcal{Q})}
    -
        3L_n\delta_{n,N}
    \right)\\
    &\geq
    \frac{\lambda_1(-\tilde{\Delta}_{\overline{r}_n, \mathfrak{Q}^m(1/N)})}
    {2^m}
    \left(
       \tilde{\alpha}
    -
        3L_n\delta_{n,N}
    \right).
\end{align}
By Lemma \ref{lem:firstEigenvalueOfCubes}, we choose $N$ big enough to ensure that for any $n$,
\begin{equation}
    \lambda_1(-\tilde{\Delta}_{\overline{r}_n, \mathfrak{Q}(1/N)})
>
    C \tilde{C}\frac{2^{m+2}}{\tilde{\alpha}} \, \cdot
\end{equation}
By the weak convergence \eqref{eq:hnto0} we know that $\delta_{n,N} \rightarrow 0$, and so we can choose $n$ big enough to guarantee
\begin{equation}
    \delta_{n,N} < \frac{\tilde{\alpha}}{6L_N} \, \cdot
\end{equation}
With these choices we eventually get
\begin{equation}
    \tilde{E}_{r_n}(f_{k,r_n}-f)
>
    C,
\end{equation}
which is a contradiction.

\textbf{Step 2.} Now we show that $\tilde{\lambda}_{k,r_n} \rightarrow \mu_k$, where $\mu_k$ is the $k$-th Neumann eigenvalue. Let $r_n \rightarrow 0$. We know by Proposition \ref{weakConvergenceLemma} that there exist eigenfunctions $f_0,...,f_k$ with Neumann eigenvalue $\lambda_0,...,\lambda_k$ such that
        \begin{equation}
            f_{i,r_n}
            \stackrel{L^2}{\rightarrow}
            f_i, \quad\quad \forall i \in \{0,...,k\},
        \end{equation}
        \begin{equation}
            \tilde{\lambda}_{k,r_n}
        \rightarrow
            C_m \lambda_k,
        \end{equation}
        and
        \begin{equation}
        \label{eq:ineqBetweenEigenvalues}
            \lambda_i \leq \lambda_k
        \quad\quad \forall i \in \{0,...,k\}.
        \end{equation}
        Since $\langle f_{i,r_n}, f_{j,r_n}\rangle = \delta_{i,j}$, we also have by strong convergence that $\langle f_i, f_j\rangle = \delta_{i,j}$. Thus we have that
        \begin{equation}
            V_{k+1} \df \mathrm{Span}(f_0,\ldots,f_k) \in \mathcal{G}_{k+1}(L^2(M,\vol_g)),
        \end{equation}
        and so by equation \eqref{eq:ineqBetweenEigenvalues}, we conclude
        \begin{equation}
            C_m \mu_k
        \leq
            \max_{f \in V_{k+1}}
            \frac{\langle
                \nabla f, \nabla f
            \rangle}
            {
                \|f\|_{L^2}
            }
        =
            C_m\lambda_k
        =
            \lim_n \tilde{\lambda}_{k,r_n}.
        \end{equation}
        This shows that $\liminf_{r \rightarrow 0} \tilde{\lambda}_{k,r_n} \geq C_m\mu_k$

        To prove $\limsup_{r \rightarrow 0} \tilde{\lambda}_{k,r} \leq C_m \mu_k$, let $\{f_0,\ldots,f_k\}$ be an $\langle \cdot,\cdot \rangle_2$-orthonormal family of Laplace (resp.~Neumann) eigenfunctions associated with the eigenvalues $\{\mu_0,\ldots,\mu_k\}$ respectively satisfying $\mu_0 \leq ... \leq \mu_k$. By elliptic regularity, we know that these functions belong to $C^\infty(M)$. Then Proposition \ref{convergenceOfInnerProds} implies that given $\epsilon>0$, there exists $r_\epsilon>0$ such that for $r \in (0,r_\epsilon)$,
    \begin{equation}
    \left|
        \langle 
            -\tilde{\Delta}_r f_i,f_j 
        \rangle_2
    -
        \delta_{i,j}C_m\mu_j
    \right|
    <
    \epsilon
    \end{equation} 
    where $\delta_{i,j}$ is the usual Kronecker delta. Set $U \df \mathrm{Span}\left(f_0,...,f_k\right)$ and
    \begin{equation}
        v \df \sum_{i=1}^ka_i\psi_i
    \end{equation}
    for some $a=(a_1,\ldots,a_k) \in \mathbb{S}^{k-1}$.
    Then
    \begin{align}
    \left|
        \langle 
            -\tilde{\Delta}_r v,v 
        \rangle
    -
        \sum_{i=1}^k
            a_i^2C_m\mu_i
    \right|
    =
    \left|
        \sum_{i,j=1}^k
        a_ia_j
        \langle 
            -\tilde{\Delta}_r f_i,f_j 
        \rangle
    -
        \sum_{i=1}^k
            a_i^2C_m\mu_i
    \right|
    \leq
    k^2\epsilon.
    \end{align}
    Since $U$ is a $k+1$-dimensional subspace, we conclude that
    \begin{equation}
        \tilde{\lambda}_{k,r}
    \leq
        \max_{v\in U}
        \frac{\langle 
            -\tilde{\Delta}_r v,v 
        \rangle}{\|v\|_{2}^2}
    \leq
        \max_{a \in \mathbb{S}^{k}}
        \sum_{i=1}^{k} a_i^2\mu_i
        +
        k^2\epsilon
    \le
        \mu_k
        +
        k^2\epsilon.
    \end{equation}
    Take the limit superior as $r \to 0$ and then let $\epsilon \to 0$ to obtain $\limsup_{r \rightarrow 0} \tilde{\lambda}_{k,r} \leq C_m \mu_k$. Combined with Corollary \ref{cor:min}, the latter implies the existence of $r_k>0$ such that $\min \sigma_{\text{ess}}(-\tilde{\Delta}_r) \big) \ge \mu_k+1 \ge \tilde{\lambda}_{k,r}$ for any $r \in (0,r_k)$, so that $\tilde{\lambda}_{k,r}$ indeed coincides with $\lambda_k(-\tilde{\Delta}_r)$.
\end{proof}

\bibliographystyle{alpha}
\bibliography{biblio}

\begin{thebibliography}{BQWZ12}

\bibitem[AB18]{AB}
Luigi Ambrosio and J{\'e}r{\^o}me Bertrand.
\newblock D{C} calculus.
\newblock {\em Mathematische Zeitschrift}, 288:1037--1080, 2018.

\bibitem[AKS22]{AKS1}
Tomasz Adamowicz, Antoni Kijowski, and Elefterios Soultanis.
\newblock Asymptotically mean value harmonic functions in doubling metric
  measure spaces.
\newblock {\em Analysis and Geometry in Metric Spaces}, 10(1):344--372, 2022.

\bibitem[AKS23]{AKS2}
Tomasz Adamowicz, Antoni Kijowski, and Elefterios Soultanis.
\newblock Asymptotically mean value harmonic functions in subriemannian and
  {R}{C}{D} settings.
\newblock {\em The Journal of Geometric Analysis}, 33(3):80, 2023.

\bibitem[Ald18]{Aldaz1}
Jesus~M. Aldaz.
\newblock Local comparability of measures, averaging and maximal averaging
  operators.
\newblock {\em Potential Analysis}, 49(2):309--330, 2018.

\bibitem[BBI01]{BBI}
Dmitri Burago, Yuri Burago, and Sergei Ivanov.
\newblock {\em A course in metric geometry}, volume~33 of {\em Graduate Studies
  in Mathematics}.
\newblock American Mathematical Society, Providence, RI, 2001.

\bibitem[BQWZ12]{Belkin+}
Mikhail Belkin, Qichao Que, Yusu Wang, and Xueyuan Zhou.
\newblock Toward understanding complex spaces: {G}raph laplacians on manifolds
  with singularities and boundaries.
\newblock In {\em Conference on learning theory}, pages 36--1. JMLR Workshop
  and Conference Proceedings, 2012.

\bibitem[CC00]{CC}
Jeff Cheeger and Tobias~H. Colding.
\newblock On the structure of spaces with {R}icci curvature bounded below.
  {III}.
\newblock {\em J. Differential Geom.}, 54(1):37--74, 2000.

\bibitem[Cha95]{Chavel}
Isaac Chavel.
\newblock {\em Riemannian geometry: a modern introduction}.
\newblock Number 108. Cambridge university press, 1995.

\bibitem[Gig15]{Gigli}
Nicola Gigli.
\newblock On the differential structure of metric measure spaces and
  applications.
\newblock {\em Mem. Amer. Math. Soc.}, 236(1113):vi+91, 2015.

\bibitem[HKST15]{HKST}
Juha Heinonen, Pekka Koskela, Nageswari Shanmugalingam, and Jeremy~T. Tyson.
\newblock {\em Sobolev spaces on metric measure spaces. An approach based on
  upper gradients}, volume~27 of {\em New Mathematical Monographs}.
\newblock Cambridge University Press, Cambridge, 2015.

\bibitem[KMS01]{KMS}
Kazuhiro Kuwae, Yoshiroh Machigashira, and Takashi Shioya.
\newblock Sobolev spaces, {L}aplacian, and heat kernel on {A}lexandrov spaces.
\newblock {\em Mathematische Zeitschrift}, 238(2):269--316, 2001.

\bibitem[Kok06]{K}
Simon~L. Kokkendorff.
\newblock A {L}aplacian on metric measure spaces.
\newblock Unpublished note, 2006.

\bibitem[KS93]{KS}
Nicholas~J. Korevaar and Richard~M. Schoen.
\newblock Sobolev spaces and harmonic maps for metric space targets.
\newblock {\em Communications in Analysis and Geometry}, 1:561--659, 1993.

\bibitem[MT20]{MT1}
Andreas Minne and David Tewodrose.
\newblock Asymptotic mean value {L}aplacian in metric measure spaces.
\newblock {\em Journal of Mathematical Analysis and Applications},
  491(2):124330, 2020.

\bibitem[MT23]{MT2}
Andreas Minne and David Tewodrose.
\newblock Symmetrized and non-symmetrized {A}symptotic {M}ean {V}alue
  {L}aplacian in metric measure spaces.
\newblock {\em Proceedings of the Royal Society of Edinburgh Section A:
  Mathematics}, pages 1--38, 2023.

\bibitem[RS80]{RS}
Michael Reed and Barry Simon.
\newblock {\em Methods of modern mathematical physics, 1. {F}unctional
  {A}nalysis}.
\newblock New York, London: Academic Press, 1980.

\bibitem[Str74]{Struble}
Raimond~A Struble.
\newblock Metrics in locally compact groups.
\newblock {\em Compositio Mathematica}, 28(3):217--222, 1974.

\bibitem[Vil09]{Villani}
C\'{e}dric Villani.
\newblock {\em Optimal transport}, volume 338 of {\em Grundlehren der
  Mathematischen Wissenschaften [Fundamental Principles of Mathematical
  Sciences]}.
\newblock Springer-Verlag, Berlin, 2009.
\newblock Old and new.

\bibitem[WS72]{WeinsteinStenger}
Alexander Weinstein and William Stenger.
\newblock {\em Methods of intermediate problems for eigenvalues: theory and
  ramifications}, volume~89.
\newblock Elsevier, 1972.

\end{thebibliography}

\end{document}